\documentclass[11pt,reqno]{amsart}
\usepackage{latexsym,amsmath}

\usepackage[left=3cm,top=2.5cm,right=3cm,bottom=2.5cm]{geometry}
\usepackage{amsthm}
\usepackage{amssymb}

\usepackage{hyperref}

\hypersetup{colorlinks=true,
	citecolor=blue,
	filecolor=blue,
	linkcolor=blue,
	urlcolor=blue
}
\usepackage{cleveref}

\usepackage{epsfig}
\usepackage{enumitem}
\usepackage[normalem]{ulem}

\def\PP{\mathcal{P}}

\def\p{\textprime}
\newcommand{\hm}[1]{\leavevmode{\marginpar{\tiny%
$\hbox to 0mm{\hspace*{-0.5mm}$\leftarrow$\hss}%
\vcenter{\vrule depth 0.1mm height 0.1mm width \the\marginparwidth}%
\hbox to 0mm{\hss$\rightarrow$\hspace*{-0.5mm}}$\\\relax\raggedright #1}}}

 \usepackage[usenames,dvipsnames]{color}
\usepackage[normalem]{ulem}

\begin{document}
\allowdisplaybreaks[2]

\newtheorem{theorem}{Theorem}[section]
\newtheorem{cor}[theorem]{Corollary}
\newtheorem{lemma}[theorem]{Lemma}
\newtheorem{fact}[theorem]{Fact}
\newtheorem{property}[theorem]{Property}
\newtheorem{corollary}[theorem]{Corollary}
\newtheorem{proposition}[theorem]{Proposition}
\newtheorem{claim}[theorem]{Claim}
\newtheorem{conjecture}[theorem]{Conjecture}
\newtheorem{definition}[theorem]{Definition}
\theoremstyle{definition}
\newtheorem{example}[theorem]{Example}
\newtheorem{remark}[theorem]{Remark}
\newcommand\eps{\varepsilon}
%\parindent{0pt}
\parindent0pt

%\documentclass[12pt]{llncs}

%\usepackage{amssymb}
%\usepackage{amsmath}
%\usepackage{cmmib57}

%\usepackage{makeidx}
%\DeclareMathAlphabet{\Sss}{U}{bbmss}{m}{n}
%\def\Ex{{\Sss E}}
%\def\Z{{\Sss Z}}
%\usepackage{amssymb}
%\newcommand{\ignore}[1]{{}}
%\newcommand{\dR}{\mathbb R}
%\newcommand{\de}{\mbox{d}}

%\def\N{N}

\title{Canonical theorems for colored integers with respect to some linear combinations}
%\eee{change title or leave it as it is ?}}

%\titlerunning{title} 

\author{Maria Axenovich}
\address{Karlsruhe Institute of Technology, Institute of Algebra and Geometry, Englerstra\ss{}e 2, 76131 Karlsruhe, Germany}
\email{maria.aksenovich@kit.edu}

\author{David S. Gunderson}
\address{Department of Mathematics, University of Manitoba, Winnipeg, Manitoba, Canada R3T 2N2}
\email{gunderso@cc.umanitoba.ca}

\author{Hanno Lefmann}
\address{TU Chemnitz, Fakult\"at f\"ur Informatik, Stra\ss{}e der Nationen 62, 09107 Chemnitz, Germany}
\email{lefmann@informatik.tu-chemnitz.de}

%\maketitle

\date{\today}

\begin{abstract} \small\baselineskip=9pt 
Hindman proved in 1979 that no matter how natural numbers are colored in $r$ colors, for a fixed positive integer $r$, 
there is an infinite subset $X$ of numbers  and a color $t$  such that for any finite {non-empty} subset $X'$ of $X$,  the color of the  sum of elements  from $X'$ is $t$.
Later, {Taylor} extended this result to colorings with unrestricted number of colors and {five} unavoidable color patterns on finite sums. 
This result is referred to as a {\it canonization} of {Hindman's} theorem and parallels the Canonical Ramsey Theorem {of Erd\H{o}s and Rado}.
We extend   Taylor's result from sums, that are linear combinations with coefficients $1$,  to several linear combinations with  coefficients $1$ and $-1$. 
These results in turn could be interpreted as canonical-type theorems for solutions  to  infinite systems.

\end{abstract}

\maketitle

\markboth{}{}

\thispagestyle{empty}

\section{Introduction}

Ramsey-type questions are concerned with partitions of various discrete structures into parts, often associated with color classes, and finding unavoidable patterns in such partitions.
While classical Ramsey-type results are restricted to partitions into a fixed number of classes and only monochromatic patterns, a so-called {\it canonisation} deals with partitions into arbitrary number of classes and various unavoidable patterns, referred to as {\it canonical situations}.  One of the easiest examples is given for a positive integer $n$  by  coloring a  set  {$(n-1)^2+1$ elements and observing that there is an $n$-element subset that satisfies one of the canonical situations - it is  either monochromatic, i.e., has all elements of the same color or  it is  rainbow, i.e., having all elements of distinct colors.   In this paper we consider colorings of natural numbers into arbitrary number of colors and unavoidable color patterns on certain linear combinations. When linear combinations are  sums of elements, the following classical theorems give us unavoidable patterns in the case of the fixed number of colors and arbitrary number of colors.

{Let ${\mathbb N}$ be the set of positive integers.  Let $\PP'= \PP'({\mathbb N})$ be the set of all finite non-empty subsets of ${\mathbb N}$.} For $r \in {\mathbb N}$ let $[r] = \{1,2, \ldots , r\}$.
%\eee{May be ${\mathbb N}'$ instead of $\PP'$?}

\begin{theorem} [Hindman  \cite{hindman}]\label{theo.a}
Let $r$ be a fixed positive integer. Then, for every coloring $\Delta\colon
{\mathbb N} \longrightarrow [r]$ there exist infinitely
many positive
integers $x_1 < x_2 < \cdots$, such that all their finite, non-empty sums (without repetition) $\sum_{i \in I} x_i$ 
are of the same color, $I {\in \PP'}$.
\end{theorem}

\begin{theorem}[Taylor~\cite{taylor}] \label{theo.a.canonical} 
For every coloring $\Delta\colon
{\mathbb N} \longrightarrow {\mathbb N}$ there exist infinitely 
many positive
integers $x_1 < x_2 < \cdots$, such that one of the following holds:
\begin{itemize} 
	\item[(i)] $\Delta (\sum_{i\in I } x_i) = \Delta (\sum_{j\in J } x_j)$ for all  $I, J \in \PP' $, or
	\item[ (ii)] $\Delta (\sum_{i\in I } x_i) = \Delta (\sum_{j\in J } x_j)$ if and only if $I=J$,
for all  $I, J \in \PP'$, or
\item[ (iii)]  $\Delta (\sum_{i\in I } x_i) = \Delta (\sum_{j\in J } x_j)$ if and only if $\mbox{max } I = \mbox{max } J$,
for all $I, J \in \PP'$, or
\item [(iv)] $\Delta (\sum_{i\in I } x_i) = \Delta (\sum_{j\in J } x_j)$ if and only if $\mbox{min } I = \mbox{min } J$, for all $I, J \in \PP'$, or
\item[ (v) ] $\Delta (\sum_{i\in I } x_i) = \Delta (\sum_{j\in J } x_j)$ if and only if  $\mbox{min } I = \mbox{min } J$ and $\mbox{max } I = \mbox{max } J$,
for all  $I, J \in \PP'$.
\end{itemize}
None of these five patterns may be omitted without violating the theorem.
\end {theorem}

{ For any set $X$ we denote the set of all $k$-element subsets of $X$ by 
$\binom{X}{k}$.}
We are extending Theorem \ref{theo.a.canonical} from arbitrary finite sums to finite sums with a given number of summands and to some linear combinations with coefficients equal to $1$ or $-1$. We consider arbitrary colorings of natural numbers.  For an infinite set $X$  of natural numbers we consider special sets of linear combinations:
\begin{enumerate}
\item  $X\cup \{x_1+x_2 + \cdots +x_k:~~ x_1, x_2, \ldots, x_k \in X,~ x_1<x_2<\cdots < x_k \}$ in Section~\ref{x,x1+x2+xk},\\
\item  $X \cup \{ x_k-x_{k-1} + x_{k-2}- x_{k-3}+ \cdots +x_2-x_1: ~ x_1, x_2, \ldots, x_k \in X, ~ x_1<x_2< \cdots < x_k\}$, $k$~even, in Section  \ref{section_even},\\
\item $X \cup \{ x_{k} - x_{k-1} + \cdots + x_{3} - x_{2} + x_{1}: ~ x_1, x_2, \ldots, x_k \in X,   ~ x_{1}< x_{2} < \cdots < x_{k} \}$, $k\geq 3$~odd, in Section~\ref{X-alternating-odd},\\
\item $\{x_1+x_2 + \cdots +x_k: ~ x_1, x_2, \ldots, x_k \in X, ~ x_1<x_2 < \cdots<x_k \}$ in Section~\ref{x1+x2+xk},\\
\item  $\{ x_k-x_{k-1} + x_{k-2}- x_{k-3}+ \cdots +x_2-x_1: ~ x_1, x_2, \ldots, x_k \in X, ~ x_1<x_2< \cdots < x_k\}$, $k$ even, in Section \ref{section_even,ell},\\
\item $ \{ x_{k} - x_{k-1} + \cdots + x_{3} - x_{2} + x_{1}: x_1, x_2, \ldots, x_k \in X, ~  x_{1}< x_{2} < \cdots < x_{k} \}$, $k\geq 3$ odd, in Section ~\ref{alternating-odd}.
\end{enumerate}

For each of these items, we describe  unavoidable canonical situations. I.e., no matter how one colors natural numbers, there is  {an infinite}  subset $X=\{x_1<x_2< \cdots \}$ such that the corresponding  linear combinations satisfy one of the described canonical situations. \\
The existence of canonical situations in items (1) and (4) follow immediately from Theorem \ref{theo.a.canonical}. For the item (1) we show that the number of canonical situations could be reduced if $k=2$, otherwise five necessary situations remain. For item (4) we show that there are three necessary canonical situations. \\

Note that there is a qualitative difference between the first three items we list, that include color conditions on $X$, and the last three items, where the color conditions are only on the linear combinations with the same, say $k$ number of terms.  In fact, the following Ramsey-type result on $k$-tuples provides us with canonical patterns when we define the color of a tuple by the color of the respective linear form.\\

For sets $X= \{x_1 < x_2 < \cdots  < x_k\}$ and $I\subseteq \{1, 2, \ldots , k\}$ let $X:I = \{ x_i:  \, i \in I \}$.
For a set $Y$ and a function $\Delta$ defined on the elements of $Y$ we write  $\Delta(Y) = \{\Delta(y): y\in Y\}$.

\begin{theorem} [Erd\H{o}s and Rado~\cite{ER50}] \label{can_ramsey}
Let $k$ be a positive integer. For every coloring $\Delta \colon {\binom{\mathbb N}{k}} \longrightarrow
{\mathbb N}$ of the set  of $k$-element subsets of
 ${\mathbb N}$
there exists a subset  $I\subseteq \{1, 2,\ldots , k\}$ and an infinite subset $X \subset {\mathbb N}$
such that for all $k$-element sets $Y, Z \in {\binom{X}{k}}$ 
$$
\Delta (Y) = \Delta(Z)\; \;  \mbox{   if and only if    } \; \; Y:I = Z:I.
$$
None of these $2^k$ patterns may be omitted without violating the theorem.
\end{theorem}

Thus, for  the  items  (4)-(6) that  we consider one can always provide a set of at most $2^k$ canonical patterns.  Instead of $2^k$, we could find a sets of three, five, and three canonical patterns for items (4), (5), and (6), respectively. \\

Before we present  our results in Sections~\ref{sums}-\ref{alternating-odd}, we describe a connection between the considered problem and solutions to systems of linear algebraic equations with respect to a given partition of natural numbers in Section~\ref{systems}.  The last Section~\ref{conclusions} states conclusions and open problems.

\vskip 1cm

%%%%%%%%%%%%%%%%%%%%%%%%%%%%%%%%%%%%%%%%%%%%%%%%%
%%%%%%%%%%%%%%%%%%%%%%%%%%%%%%%%%%%%%%%%%%%%%%%%%
\section{Connection to systems and regular partitions}\label{systems}
%%%%%%%%%%%%%%%%%%%%%%%%%%%%%%%%%%%%%%%%%%%%%%%%%
%%%%%%%%%%%%%%%%%%%%%%%%%%%%%%%%%%%%%%%%%%%%%%%%%

Let $A=(a_{i,j})$, $i\geq 1, j\geq 1 $ be an integer-valued   infinite matrix with each row
 containing only finitely many non-zero entries. Call a homogeneous
 system $Ax = A(x_1, x_2, \ldots)^T = 0$
of linear equations {\em partition regular in ${\mathbb N}$}
 if and only if for every positive
 integer $r$ and
any coloring $\Delta\colon {\mathbb N} \longrightarrow [r]$ of 
 ${\mathbb N}$ with $r$ colors there exist positive integers
$y_1, y_2, \ldots$ with $\Delta(y_1) = \Delta (y_2) = \cdots $ such that
 $A(y_1, y_2, \ldots)^T = 0$. Partition regularity for finite matrices is defined similarly.

The theory of finite partition
 regular homogeneous systems of  linear equations has been  studied in particular 
 by Rado~\cite{Rado33} and Deuber~\cite{Deuber73}. \\

 An integer-valued $N \times M$-matrix $A$ has the {\em columns property} if
and only if  the set $\{ 1,2, \ldots , M\}$ of column indices
 can be partitioned as
$\{ 1,2, \ldots , M\} = I_0 \cup I_1 \cup \cdots \cup I_m$
such that (i) the sum of all columns with indices in
 $I_0$ add up to the all-zero vector,
 and  (ii) the sum of all columns with indices 
in $I_j$ is a rational
linear combination of all columns with  indices in $I_0 \cup \cdots \cup
I_{j-1}$,  for $j=1, 2, \ldots,m$.

 \begin{theorem}[Rado  ~\cite{Rado33}]\label{rado_thm}
Let $A$ be a finite integer-valued matrix. The finite system of linear equations  $A (x_1, x_2,\ldots , x_n)^T = 0$
is  partition regular in ${\mathbb N}$ if and only if the matrix $A$
has the columns property.
\end{theorem}

For a set $S$ whose elements are colored, call $S$ {\em rainbow} if all its elements are colored distinctly.

For finite systems of linear equations, where colorings of ${\mathbb N}$ may be arbitrary, a canonical version of Rado's theorem was given in~\cite{lef8}.

\begin{theorem} [Lefmann~\cite{lef8}]\label{theo_can_equation}
Let $A$ be an 	 $N \times M$-matrix, that has the columns property with  
corresponding partition $\{ 1, 2,\ldots , M\} = I_0 \cup I_1 \cup \cdots \cup I_m$ of
the set $\{1,2, \ldots , M\}$ of column indices. Then, for every coloring  $\Delta\colon {\mathbb N} \rightarrow {\mathbb N}$
 there exist positive integers
$y_1, y_2, \ldots, y_M$ such that  
 $~A(y_1, y_2,  \ldots, y_M)^T = 0$, and one of the following holds:
 \begin {itemize}
 \item[(i)] the set $\{y_1, y_2,\ldots , y_M\}$ is monochromatic, or
 \item[(ii)] the set $\{y_1, y_2, \ldots , y_M\}$ is rainbow, or
 \item[(iii)] $\Delta (y_i) = \Delta (y_j)$ if and only if $i,j\in I_k$,  for some $k \in \{1, 2,\ldots , m \}$,  for any $1 \leq i,j\leq M$. 
 	 \end{itemize}
 \end{theorem}
 
  If a matrix $A$ does not have the columns property, nothing is really known concerning the canonical situations. 

 Note that Theorem~\ref{theo_can_equation} provides three canonical situations~(i), (ii), and~(iii).  Moreover, one can find different sets of  canonical situations in this case. 
If a set of canonical situations does not contain (iii),  there must be at least four canonical situations  as noted in  \cite{lef8}. In this setting, one can talk about the number of canonical situations, since a canonical situation here is simply a partition of the set $\{y_1, y_2, \ldots , y_M\}$.  In some other results presented in this paper,   classes of unavoidable color patterns are not necessarily described in terms of partitions, thus we avoid quantifying the number of canonical situations.

 The proof of Theorem~\ref{theo_can_equation} in~\cite{lef8} proceeds roughly like this: one can describe a solution by pattern~(i), (ii) or (iii), and there is a coloring, that shows that patterns~(i) and (ii) do not suffice to describe the canonical situation. So there might be other descriptions of the patterns. Such a proof strategy will be used in this paper.

The  case for infinite partition regular systems was addressed  in~\cite{ghl}.

\begin{theorem}[Gunderson, Hindman and Lefmann~\cite{ghl}] \label{ghl}  Let $k$ be  a positive integer and let  $(a_1, a_2,\ldots , a_k)$ be 
 a sequence of non-zero integers, *where for each $i=1, 2,\ldots , k$ there is some $\alpha_i \in {\mathbb N} \cup \{0\}$ such that $a_i = c^{\alpha_i}$ or  $a_i = -c^{\alpha_i}$ for the same $c \in {\mathbb N} \setminus \{1 \}$*.
The infinite system of linear equations
\begin{eqnarray} \label{gen_system_**}
\langle 
a_1  x_{j_{1}} + a_2  x_{j_{2}}  +\cdots + a_k    x_{j_{k}} = x_{j_{1}, j_{2}, \cdots, 
j_{k}}; ~ 
1 \leq j_{1} < j_2 < \cdots < j_{k} \rangle
\end{eqnarray}
is partition regular in ${\mathbb N}$  {if and only if} either
(i) $a_k = 1$ and $a_1 + a_2 + \cdots + a_k =0$, or (ii)  $a_1 + a_2 + \cdots + a_k = 1$, or (iii)
$a_1 = a_2 = \cdots = a_k = 1$.
\end{theorem}

{We remark that in~\cite{ghl} it is conjectured that Theorem~\ref{ghl} holds even  if we omit the condition within the stars * there.}

 In Theorem~\ref{ghl} in cases~(i) and~(iii) the $x_i$'s can be assured to be pairwise distinct. However, in~(ii) for arbitrary sequences  $(a_1, a_2, \ldots , a_k)$ with $a_1 + a_2 + \cdots + a_k = 1$  all $x_i$'s are allowed to be the same. It is not known, for which such sequences  one can achieve  the $x_i$'s to be pairwise distinct. Clearly, for this to hold  $a_k $ must be positive.

There are sequences  $(a_1, a_2, \ldots , a_k)$ with $a_1 + a_2 + \cdots + a_k = 1$ and $a_k > 0$, where 
one cannot achieve that the $x_i$'s are pairwise distinct.
 One such sequence for $k=2$ is $(-1,2)$. Namely, consider the coloring $\Delta \colon {\mathbb N} \longrightarrow \{ 0,1,2,3,4 \}$, {where $\Delta(1) = 0$ and for $x \geq 2$} $\Delta(x) \equiv i \bmod 5$ for $ x \in \left[ \left\lceil \sqrt{2}^i\right\rceil,  \left\lceil \sqrt{2}^{i+1}\right\rceil - 1\right]$. For any infinite sequence $x_1 < x_2 < \cdots$, consider 
 $-x_1 + 2x_\ell$, $\ell \geq 2$, where $x_\ell$ is sufficiently large compared to $x_1$. Then, if $x_\ell \in  \left[ \left\lceil \sqrt{2}^i\right\rceil,  \left\lceil \sqrt{2}^{i+1}\right\rceil - 1\right]$ we have 
 $-x_1 + 2x_\ell \in \left[ \left\lceil \sqrt{2}^{i+1}\right\rceil,  \left\lceil \sqrt{2}^{i+4}\right\rceil - 1\right]$, thus $\Delta(x_\ell) \neq \Delta(-x_1 + 2x_\ell)$.\\
  A similar argument works for any sequence  $(a_1, a_2,\ldots , a_k)$ with $a_1 + a_2 + \cdots + a_k = 1$ and $a_k \geq 2$.
 However, as a consequence of our considerations concerning the canonical situation for the system~(\ref{gen_system_**}), for  alternating sequences of $-1$'s and $+1$'s, whose sums are equal to $1$, we show later, see  Corollary~\ref{cor-1,+1}, that one can achieve the $x_i$'s to be pairwise distinct.\\

Milliken~\cite{milliken} and Taylor~\cite{taylor} independently extended Hindman's Theorem~\ref{theo.a} to colorings of $k$-element sets of sets. 
Note that Milliken-Taylor type results imply the corresponding results for systems of equations. 
Indeed,  for an infinite set  $X=\{x_1<x_2<\cdots \}\subseteq {\mathbb N}$, consider  integer coefficients  $a_1, a_2, \ldots, a_k$ and  a set 
 $\mathcal{F}(X){=} X\cup \{ \sum_{i=1}^k a_ix_{j_i}: ~ x_{j_1}, x_{j_2}, \ldots , x_{j_k}\in X,~ 1 \leq  j_1< j_2 < \cdots < j_k\}$ of linear combinations. Assume we know that for any coloring of natural numbers there is an infinite 
set $X=\{x_1<x_2<\cdots \}$ such that $\mathcal{F}(X)$ satisfies one of patterns from  $P$ for a set of canonical patterns $P$.
Then, it implies that for any coloring of the natural numbers there is a solution to the infinite system 
$\langle \sum_{i=1}^k a_ix_{j_i} = x_{j_1,j_2, \ldots, j_k}:~ 1\leq j_1<j_2<\cdots <j_k\rangle$ of equations satisfying one of the color patterns from $P$. 
However, these two formulations are not equivalent since a solution to the system allows for the values of the variables  $x_{j_i}$'s to be repeated,  that is not the case in the first formulation involving the family $\mathcal{F}(X)$ of linear forms.\\

In this paper we focus on the first, stronger formulation in terms of linear forms and consider several specific situations with $a_i\in \{1, -1\}$, $i=1, 2, \ldots, k$.

\vskip 1cm

%%%%%%%%%%%%%%%%%%%%%%%%%%%%%%%%%%%%%%%%%%%%%%%%%%%%%%%%%%%%%%%%%%%%%
%%%%%%%%%%%%%%%%%%%%%%%%%%%%%%%%%%%%%%%%%%%%%%%%%%%%%%%%%%%%%%%%%%%%%
%%%%%%%%%%%%%%%%%%%%%%%%%%%%%%%%%%%%%%%%%%%%%%%%%%%%%%%%%%%%%%%%%%%%%

\section{The set $X\cup \{x_1+x_2 + \cdots +x_k:~~ x_1, x_2, \ldots, x_k \in X,~ x_1<x_2<\cdots < x_k \}$}\label{x,x1+x2+xk} \label{sums}

%%%%%%%%%%%%%%%%%%%%%%%%%%%%%%%%%%%%%%%%%%%%%%%%%%%%%%%%%%%%%%%%%%%%%
%%%%%%%%%%%%%%%%%%%%%%%%%%%%%%%%%%%%%%%%%%%%%%%%%%%%%%%%%%%%%%%%%%%%%
%%%%%%%%%%%%%%%%%%%%%%%%%%%%%%%%%%%%%%%%%%%%%%%%%%%%%%%%%%%%%%%%%%%%%

The next result is an analogue of Taylor's theorem~\ref{theo.a.canonical} for $k$-term sums with $k \geq 2$ fixed.

\begin{theorem} \label{theo*.aa}
Let $k\geq 2$ be an integer. Let $\Delta \colon {\mathbb N} \longrightarrow {\mathbb N}$ be an arbitrary 
coloring.
Then, one can find an infinite set $X= \{ x_1 < x_2 < \cdots\}$,
such that the sets $X$ and  
$X_{sum}= X_{sum}(k)=   \{ (\sum_{j\in J } x_j):~ J \in \binom{\mathbb N}{k} \}$  are  colored according to one of the following
patterns:
\begin{itemize}
\item[(i)] $X \cup X_{sum}$ is monochromatic, or 
\item[(ii)]  $X \cup X_{sum}$ is rainbow, or 
\item[(iii)] $X$ is rainbow and $\Delta(\sum_{j\in J } x_j)= \Delta(x_{\max J})$ for all $J\in \binom{\mathbb N}{k}$, or 
\item[(iv)]  $X$ is rainbow and $\Delta(\sum_{j\in J } x_j)= \Delta(x_{\min J})$ for all $J\in \binom{\mathbb N}{k}$,  or
 \item[(v)]   $X$ is rainbow, and  $\Delta(\sum_{j\in J } x_j)= \Delta (\sum_{j\in I} x_j)$ if and only if $\min I=\min J$ and $\max I=\max J$,  for all $I,J\in \binom{\mathbb N}{k}$, and $\Delta(X) \cap \Delta(X_{sum})= \emptyset$.
\end{itemize}
Moreover, patterns~(ii) and~(v) coincide for $k=2$ and none of the patterns may be omitted without violating the theorem.
\end{theorem}

\begin{proof}
The sufficiency follows from Taylor's Theorem~\ref{theo.a.canonical}. Let $k \geq {2}$. Next it will be shown that none of these five patterns may be omitted without violating the theorem.
Clearly, patterns~(i) and~(ii) must be there by considering a monochromatic and a rainbow coloring of the positive integers.\\

We shall show that none of the pattern~(iii)-(v) may be omitted by considering two colorings  $\Delta_1$ and $\Delta_2$  of ${\mathbb N}$:\\

$\Delta_1 (x) = i$ if and only if $x \in I_i= [k^i, k^{i+1}-1]$, $i=0,1, \ldots$, and \\
$\Delta_2 (x) = x'$, where $x= k^{x'} x''$, and $x' \in {\mathbb N} \cup \{0\}$ and $x''\in {\mathbb N}$ is not divisible by $k$. \\

Note that one of the patterns (i)-(v) is satisfied for an infinite sequence $x_1<x_2< \cdots$ if and only if this pattern is satisfied on each infinite subsequence of $x_1<x_2<\cdots$. 
We shall consider an arbitrary infinite sequence $x_1<x_2< \cdots$ of positive integers and, by taking subsequences (later we also refer to this as {\it thinning}) assume that $x_i$ belongs to  $[k^{k_i}, k^{k_i+1}-1]$, where $k_1<k_2<\cdots$.
In particular we can assume that $\Delta_1$ is rainbow on $x_1, x_2, \ldots$.
Furthermore, we see that there is either an infinite set of $x_i$'s that have the same value of $\Delta_2(x_i)$ or an infinite set of $x_i$'s such that $\Delta_2(x_i)$'s are pairwise distinct.
Considering infinite subsequences again, we can assume again that one of the two cases  holds:\\

Case~(a): $\Delta_2$ is monochromatic on $x_1, x_2, \ldots $, or \\
Case~(b): $\Delta_2$ is rainbow on $x_1 < x_2 < \ldots $, moreover $\Delta_2(x_1)<\Delta_2(x_2) <\cdots$.\\

Moreover, by taking subsequences  we can assume that there are integer constants  ${c,d}$, with $1\leq d <k$ and $0\leq c <k$,   such that if $x_i= k^{x_i'}x_i''$,  for non-negative integers $x_i', x_i''$  and $x_i'$ largest possible, then  $x_i'' \equiv d + ck  \bmod {k^2}$ for all $i=1, 2, \ldots$. \\

{\bf Claim 1.}  Pattern~(iii) can not be omitted.\\
Suppose that one can omit pattern~(iii). Consider the coloring $\Delta_1$. Since it is rainbow,  it does not have pattern~(i). 
Take  $x_{3k}$ and assume that $\Delta_1(x_{3k})=m$. 
Consider the set $Y= \{\sum_{i\in I} x_i + x_{3k}:   I= \{j, j+1\ldots, k+j-1\}, ~ j=1,2,3\}$.
Then for any $y\in Y$, $y \in [k^m+1, k^{m+2}-1]$, thus $\Delta_1(y) \in \{m, m+1\}$. Since $|Y|=3$, some two elements of $Y$ have the same color under $\Delta_1$.  In particular, patterns~(ii) and~(v)  are  not satisfied. 
 Moreover, we do not have pattern~(iv), as { $\Delta_1( \sum_{i=1}^k x_i ) \neq  \Delta_1(\sum_{i=1}^{k-1} x_i + x_{3k})$}.
  Thus, the only remaining pattern is (iii).\\

{\bf Claim 2. } Pattern~(iv) can not be omitted.\\
Consider the coloring $\Delta_2$. 
If case~(a) holds, then {for some non-negative integer $x'$,}  $\Delta_2(x_i)=x'$, for all $i$.  Thus $x_i= k^{x'}(d+kx_i'')$,  $i=1, 2,\ldots$. Then  $\sum_{i=1}^k x_i = k^{x'}[(d+kx_1'') + (d+kx_2'')+ \cdots + (d+ kx_k'')]= k^xq $, {for positive integers $x, q$, } where $x> x'$.  Thus $\Delta_2(x_{1}+x_2 + \cdots + x_k) \neq \Delta_2(x_1)$ and  pattern~(i) is not satisfied.
Since in  case~(a)  $\Delta_2$ is monochromatic on $x_1, x_2, \ldots $, none of patterns~(ii)-(v) hold. 
Thus we can assume that case~(b) holds, i.e., that $\Delta_2$ is rainbow on $x_1, x_2, \ldots $. {Thus pattern~(i) does not occur. 
Consider $\sum_{i=1}^k x_i = \sum_{i=1}^k k^{x_i'}(d+kx_i'')$, where $x_1'<x_2'<\cdots <x_k'$. Then 
$\sum_{i=1}^k x_i = k^{x_1'} (d + kx_1'' + \sum_{i=2}^k k^{x_i'-x_1'}(d+kx_i''))$, so $\Delta_2(\sum_{i=1}^k x_i)= x_1' = \Delta_2(x_1)$.
Thus, } patterns~(ii), (iii), and (v)  cannot occur because  $\Delta_2(x_k)\neq \Delta_2(x_1) = \Delta_2 ( \sum_{i=1}^k x_i)$.\\

{\bf Claim 3.} Pattern~(v) can not be omitted.\\
Note that if $k=2$, patterns~(v) and~(ii) coincide, and we already showed that pattern~(ii) can not be omitted. So, assume that $k\geq 3$.
Consider the coloring $\Delta_3$ of ${\mathbb N}$ defined by $\Delta_3(x)= (\Delta_1(x), \Delta_2(x))$. 
Since $\Delta_1$ and $\Delta_2$ do not satisfy pattern (i), so does not $\Delta_3$. Pattern~(iv) does not occur as {$\Delta_1(x_1) \neq \Delta_1( \sum_{i=1}^{k-1} x_i + x_{2k})$}.  {It remains to verify that patterns (ii) and (iii) do not occur.}\\

 Assume that case~(a) holds,  {i.e.,  $\Delta_2(x_i)=x'$ and $x_i= k^{x'}(d+kx_i'')$, for all $i$.}  {Consider 
  $\sum_{i=1}^k x_i = k^{x'}\sum_{i=1}^k (d+ kx_i'') = k^{x'} (dk + \sum_{i=1}^k kx_i'')  = k^{x'+1} ( d + \sum_{i=1}^k x_i'') $, then in particular $\Delta_2(\sum_{i=1}^k x_i) >x' =\Delta_2(x_k)$,}  and pattern ~(iii)   does not occur.
 From an argument used in Claim~1 we see that there are two sets 
   $I = \{j_1, j_1+1, \ldots, k+j_1 - 1\}$ and $I' = \{j_2, j_2+1, \ldots, k+j_2 - 1\}$,  $j_1 \neq j_2$,  such that  for  $y= \sum_{i\in I}x_i + x_{3k}$ and $y'=\sum_{i\in I'}x_i + x_{3k}$, we have  $\Delta_1(y)= \Delta_1(y')$. Next we shall show that  $\Delta_2(y)= \Delta_2(y')$. 
  Indeed, we know that $x_i = k^{x'}( d + ck + \ell_i k^2)$, for some non-negative integers $\ell_i$'s, for all $i$.
  Then $y = k^{x'} ( dk + ck^2 + {k^2}\sum_{i\in I \cup \{3k\}} \ell_i )$ and  $y' = k^{x'} ( dk + ck^2 + {k^2} \sum_{i\in I' \cup \{3k\}} \ell_i) $. Thus $\Delta_2(y)=\Delta_2(y')=x'+1$. This implies that $\Delta_3(y)= 
  (\Delta_1(y), \Delta_2(y))=   (\Delta_1(y'), \Delta_2(y'))=\Delta_3(y')$. 
Thus pattern~(ii) does not occur.\\

Assume that case~(b) holds.  Note that in this case  $\Delta_2(x_1) = \Delta_2 (x_1 + \sum_{i\in I} x_i)$ for any $I \in \binom{\{2, 3, \ldots \}}{k-1}$, and, since $k \geq 3$,    for at least two of  sets $I, I' \in \binom{\{2, \ldots, 2k-1\}}{k-2}$ the colors $\Delta_1(x_1 + (\sum_{i \in I} x_i) + x_{2k})$ and $\Delta_1(x_1 + (\sum_{i \in I'} x_i) + x_{2k})$ are the same.   Thus pattern~(ii)  does not occur.  Pattern~(iii) does not occur  as $\Delta_2(x_k) \neq  \Delta_2 (\sum_{i=1}^k x_i)$.

 Thus, none of the five patterns~(i)--(v) may be omitted without violating the theorem.
\end{proof}

\vskip 1cm

%%%%%%%%%%%%%%%%%%%%%%%%%%%%%%%%%%%%%%%%%%%%%%%%%%
%%%%%%%%%%%%%%%%%%%%%%%%%%%%%%%%%%%%%%%%%%%%%%%%%%
%%
\section{The set $X \cup \{ x_k-x_{k-1} + x_{k-2}- x_{k-3}+ \cdots +x_2-x_1: ~ x_1, x_2, \ldots, x_k \in X, ~ x_1<x_2< \cdots < x_k\}$, $k$~even} \label{section_even}
%%
%%%%%%%%%%%%%%%%%%%%%%%%%%%%%%%%%%%%%%%%%%%%%%%%%%
%%%%%%%%%%%%%%%%%%%%%%%%%%%%%%%%%%%%%%%%%%%%%%%%%%

 For finite, non-empty sets $A, B \in \PP'$ denote
$A < B$ if $\mbox{max } A < \mbox{min B}$.  For an infinite family ${\mathcal A} =\{ A_1 < A_2 < \cdots \}$
of finite subsets of ${\mathbb N}$ let
$
{\mathcal A}' = \{ \cup_{i \in I} A_i:  \; I \in \PP'
\}$,
hence ${\mathcal A}'$ is the 
 family of all finite, non-empty unions of  sets from the family 
${\mathcal A}$.

 Positive integers can be identified with
finite subsets of ${\mathbb N}\cup \{0\} $ by using the mapping
$f \colon \PP'
\longrightarrow  {\mathbb N}$ with $f(S) = \sum_{s \in S} 2^s$,
 where $S \in \PP' $.

(v\p)
Taylor's Theorem~\ref{theo.a.canonical} in terms of finite sets reads as follows.

\begin{theorem}[Taylor~\cite{taylor}] \label{theo.a.canonicaln} 
Let  ${\mathcal A} =\{ A_1 < A_2 < \cdots \}$ be  an infinite family
of finite non-empty subsets of ${\mathbb N}$.
For every coloring $\Delta'\colon
{\mathcal A}' \longrightarrow {\mathbb N}$ there exist infinitely 
many finite sets $B_1 < B_2 < \cdots$ in ${\mathcal A}'$,
 such that one of the following holds:
\begin{itemize} 
	\item[(i)] $\Delta' (\cup_{i\in I } B_i) = \Delta' (\cup_{j\in J } B_j)$ for all  $I, J \in \PP' $, or
	\item[ (ii)] $\Delta' (\cup_{i\in I } B_i) = \Delta' (\sum_{j\in J } B_j)$ if and only if $I=J$,
for all  $I, J \in \PP'$, or
\item[ (iii)]  $\Delta' (\cup_{i\in I } B_i) = \Delta' (\cup_{j\in J } B_j)$ if and only if $\mbox{max } I = \mbox{max } J$,
for all $I, J \in \PP'$, or
\item [(iv)] $\Delta' (\cup_{i\in I } B_i) = \Delta' (\cup_{j\in J } B_j)$ if and only if $\mbox{min } I = \mbox{min } J$, for all $I, J \in \PP'$, or
\item[ (v) ] $\Delta' (\cup_{i\in I } B_i) = \Delta' (\cup_{j\in J } B_j)$ if and only if  $\mbox{min } I = \mbox{min } J$ and $\mbox{max } I = \mbox{max } J$,
for all  $I, J \in \PP'$.
\end{itemize}
None of these five patterns may be omitted without violating the theorem.
\end {theorem}

Here we use the following notation for a $k$-element set $J=\{j_1<j_2< \cdots <j_k\}$ and positive integers $x_j$, $j\in J$:
$$\sum^*_{j\in J} x_j =   x_{j_k} - x_{j_{k-1}} +x_{j_{k-2}} - x_{j_{k-3}} + \cdots + x_{j_2} - x_{j_1}.$$
We refer to this as  an \emph{alternating sum}.

\begin{theorem} \label{prop-1,+1}
Let $k\geq 2$ be an even integer. Let $\Delta \colon {\mathbb N} \longrightarrow {\mathbb N}$ be an arbitrary 
coloring. 
Then, one can find an infinite set $X=\{x_1 < x_2 < \cdots\}$, such that the sets
$X$ and $X_{alt}^*= X_{alt}^*(k)=\{  \sum^*_{j\in J}x_j :  J\in \binom{\mathbb{N}}{k}\}$ are  colored according to one of the following patterns:
\begin{itemize}
\item[(i)]  $X \cup X_{alt}^*$ is  monochromatic, or
\item[(ii)]  $X \cup X_{alt}^* $  is rainbow, or
\item[(iii)]   $X$  is rainbow  and  $\Delta(\sum^*_{j\in J } x_j)= \Delta(x_{\max J})$ for all $J\in \binom{\mathbb N}{k}$, or 
\item[(iv)]  $X$ is monochromatic, and
 $\Delta(X) \cap \Delta(X_{alt}^*) = \emptyset$, and  
$\Delta( \sum^*_{j\in J} x_j)= \Delta( \sum^*_{j\in I} x_j)$ if and only if $\min I= \min J$, for all  $I,J\in \binom{\mathbb{N}}{k}$,  or
 \item[(v)]   $X$ is rainbow, $\Delta(X)\cap  \Delta(X_{alt}) = \emptyset$,  and  $\Delta(\sum^*_{j\in J } x_j)= \Delta (\sum^*_{j\in I} x_j)$ if and only if $\min I=\min J$ and $\max I=\max J$,  for all $I,J\in \binom{\mathbb N}{k}$.
\end{itemize}
For $k \geq 4$ none of the five patterns can be omitted without violating the theorem. For $k=2$ patterns~(ii) and~(v) coincide. In the latter case, none of the four patterns~(i)-(iv) can be omitted without violating the theorem.
\end{theorem}

In view of the remarks directly below Theorem~\ref{ghl} we have the following consequence of Theorem~\ref{prop-1,+1}.

\begin{corollary} \label{cor-1,+1}
Let $k,r\geq 2$ be integers, where $k$ is even. Let $\Delta \colon {\mathbb N} \longrightarrow [r]$ be a
coloring. 
Then, one can find an infinite set X= \{$x_1 < x_2 < \cdots\} \subset {\mathbb N}$, such that the set
$X \cup X_{alt}^*$ is  monochromatic.
\end{corollary}

{Note that  for the case of $k=2$, Corollary~\ref{cor-1,+1} has been shown by Rado~(\cite{Rado69}) by using Ramsey's theorem for pairs~\cite{ramsey}, i.e.,
the infinite system  $\langle x_j - x_i = x_{i,j} ; \; 1 \leq i < j  \rangle$ of linear equations is partition regular in ${\mathbb N}$.}
In our arguments we will use the following lemma.\\

Let $p\geq 3$ be a prime. For a positive integer $x$, we let $x'$ and $x''$ denote non-negative integers such that $x= p^{x'}x''$, where $x''$ is not divisible by $p$.
When the numbers are indexed, we shall also use the notation $x_{j,i}'$ for $(x_j-x_i)'$ and  $x_{j,i}''$ for $(x_j-x_i)''$, when $x_j>x_i$.

\begin{lemma}\label{lemma,2,3,4}
Let $X= \{ x_1 < x_2 < \cdots \} \subseteq {\mathbb N}$ be an infinite set. \\
There exists  an infinite set $Z = \{ z_1< z_2 < \cdots\} \subset X$, such that each of the following holds

\begin{itemize}
\item $(z_j-z_i)'' \equiv y'' \bmod p$ for some $y'' \in \{1,2,\ldots, p-1\}$, for all $i < j$, 
\item  $(z_j-z_i)' = (z_k -z_\ell)'$    if and only if $i =  \ell$, for $i<j$ and $\ell< k$, 
\item $(z_{ i+1} - z_i)' < (z_{i+2} -z_{i+1})'$ for all $i$'s.
\end{itemize}

\end{lemma}

Note that if we think of $(z_j-z_i)'$ as a color of an edge $\{i,j\}$ of an infinite complete graph, the above lemma claims that this graph is a unions of monochromatic stars $S_j$, $j=1,\ldots$,  with centers at $i$ and leaves at $j$, $j>i$, such that the color of edges in  $S_i$ is different from the color of  edges in $S_j$ for any $i\neq j$.

\begin{proof} 
We define  two  colorings 
$\Delta_4 \colon \binom{X}{2} \longrightarrow \{1,2, \ldots, p-1\}$ and 
  ~$\Delta_5 \colon \binom{X}{3} \longrightarrow \{0,1,2\}$:
\begin{eqnarray*}
\Delta_4( \{x_i < x_j\})  &   \equiv &  x_{j,i}'' \bmod p, \\
\Delta_5( \{ x_i < x_j < x_k\}) & = & 0,1, \mbox{ or }2   \hskip 0.5cm   \mbox{  if }  \hskip 0.5cm   x_{k,i}' = x_{j,i}', ~ x_{k,i}' < x_{j,i}', \mbox{ or }  x_{k,i}' > x_{j,i}', \mbox{ respectively.}\\
 \end{eqnarray*} 
Iteratively applying Ramsey's theorem~\cite{ramsey} for pairs and then for triples
we get an infinite  subset  $Y = \{ y_1< y_2< \cdots\} $ of $X$ 
such that $\binom{Y}{2}$ and $\binom{Y}{3}$
are monochromatic with respect to $\Delta_4$ and $\Delta_5$ respectively.
In particular, let  $y''$ be the color of each pair with respect to $\Delta_4$, for $y'' \in \{1,2, \ldots, p-1\}$. I.e., $y''_{j,i} \equiv y'' \bmod p$ for all $i < j$.\\

We shall show next that the set $\binom{Y}{3} $ is monochromatic in color  $0$ under $\Delta_5$. 
Indeed, the set $\binom{Y}{3} $ cannot be monochromatic in color $1$ under $\Delta_5$, as there is no infinite strictly decreasing sequence of non-negative integers.
 Assume for a contradiction, that  the set $\binom{Y}{3}$ is monochromatic in color  $2$ under $\Delta_5$.  Then  $y_{j,i}' < y_{k,i}'$ for all $i < j < k$.  
 Consider $y_1, y_2, $ and $y_j$, $j>2$. We have that $y_{2,1}'<y_{j, 1}'$, thus 
 $$y_j-y_2 = (y_j-y_1)+(y_1-y_2) = p^{y_{j,1}'} y_{j,1}'' - p^{y_{2,1}'}y_{2,1}'' = p^{y_{2,1}'}( p^{y_{j, 1}'-y_{2,1}'}y_{j,1}''  - y_{j,2}'').$$ In particular, $y_{j, 2}' = y_{2,1}'=y'$.
 On the other hand 
   $$y_j-y_1 = (y_j-y_2)+(y_2-y_1) = p^{y_{j,2}'} y_{j,2}'' + p^{y_{2,1}'}y_{2,1}'' = p^{y'}(y_{j,2}''  + y_{2,1}'').$$
  We know  that $y_{j,2}'' \equiv y'' \bmod p$ and  $y_{2,1}'' \equiv y'' \bmod p$ for odd prime $p$,  so $y_{j,2}''  + y_{2,1}'' \not\equiv 0 \bmod p$ and $y_{j,2}''  + y_{2,1}''\equiv 2y'' \bmod p \not\equiv y'' \bmod p$.
Thus   $y_{j,1}'' \not\equiv y'' \bmod p$, a contradiction. 
Thus, $\binom{Y}{3}$ is monochromatic in color $0$ under $\Delta_5$.

Since   $\binom{Y}{3} $ is monochromatic in color  $0$ under $\Delta_5$, it implies in particular  $y_{j,i}' = y_i'$ for some $y_i'$ and all $j>i$.\\

Assume that some $p+1$  values $y'_i$ coincide and are equal to $\beta$, i.e., 
for some  $I=\{i_1<i_2<\cdots<i_{p+1}\}$,  and any $q\in I$, $y_q'= \beta$.
Then we have 
$y_{i_{j+1}} - y_{i _j}= p^{\beta} y''$ for all $j=1, \ldots, p$. 
This implies that 
$y_{i_{p+1}}-y_{i_1} = (y_{i_{p+1}}-y_{i_{p}})+ \cdots + (y_{i_2}- y_{i_1}) = p^\beta ( y_{i_{p+1},i_p}'' + \cdots +y_{i_2,i_1}'')$, but $ y_{i_{p+1},i_p}'' + \cdots +y_{i_2,i_1}''  \equiv py'' \equiv 0 \bmod p$.  Thus $y_{i_{p+1}, i_1}'  > \beta$.  In particular $y_{i_{p+1}, i_1}' \neq y'_{i_2, i_1}$, a contradiction.
Thus, we can choose an infinite subset of $Y$ so that the $y_i'$'s are distinct.  
Finally, since there is no infinite decreasing sequence of non-negative integers, we can assume that  the $y_i'$s are increasing. 
Let $Z = \{ z_1<z_2<\cdots \} $ be such a subset of $Y$.
Then $Z$ satisfies the conditions of the lemma.
\end{proof}

%\vskip 1in

For the proof of Theorem \ref{prop-1,+1} and well as Theorem \ref{prop-1,+1,ell}, we shall need the following colorings and the lemma.
Consider  colorings $\Delta_1$,  $\Delta_2$, and $\Delta_3$  defined as follows for a prime $p\geq 3$:\\
\begin{itemize}
\item{}
 $\Delta_1 \colon {\mathbb N} \longrightarrow {\mathbb N}$, where 
$\Delta_1(x) = i$ if and only if $x \in [2^i, 2^{i+1} - 1]$, for an integer $i$.
\item{}
$\Delta_2 \colon {\mathbb N} \longrightarrow ({\mathbb N} \cup \{0\}) \times \{1,2, \ldots , p-1\}$ with 
$\Delta_2(x) = (x',x''\bmod p)$ if $x=p^{x'}
 x''$, where $x''$ is an integer, which is not divisible by $p$.
 \item $\Delta_3(x)= (\Delta_1(x), \Delta_2(x))$.
\end{itemize}

For a set $X=\{x_1<x_2<\cdots\}$, recall that $X_{alt}^*= X_{alt}^*(k)=\{  \sum^*_{j\in J}x_j :  J\in \binom{\mathbb{N}}{k}\}$.

\begin{lemma}\label{unavoidable-i-v}
Consider the coloring patterns defined for an infinite set $X=\{x_1<x_2<\cdots\}$ of integers:
\begin{itemize}
\item[(i\p)]  $X_{alt}^*$ is  monochromatic, 
\item[(ii\p)]  $X_{alt}^* $  is rainbow, 
\item[(iii\p)]   $\Delta(\sum^*_{j\in J } x_j)= \Delta(\sum^*_{j\in I } x_j)$  if and only if $\max I= \max J$,  for all  $ I, J\in \binom{\mathbb N}{k}$, 
\item[(iv\p)]  $\Delta( \sum^*_{j\in J} x_j)= \Delta( \sum^*_{j\in I} x_j)$ if and only if $\min I= \min J$, for all  $I,J\in \binom{\mathbb{N}}{k}$,  
 \item[(v\p)]    $\Delta(\sum^*_{j\in J } x_j)= \Delta (\sum^*_{j\in I} x_j)$ if and only if $\min I=\min J$ and $\max I=\max J$,  for all $I,J\in \binom{\mathbb N}{k}$.
\end{itemize}
For $k=2$ items (ii\p) and (v\p) coincide. For any infinite set $X=\{x_1<x_2<\cdots\}$ of natural numbers and $k\geq 2$:
\begin{itemize}
\item{} The coloring $\Delta_1$ does not satisfy (i\p ), (ii\p), (iv\p), and (v\p),
\item{} The coloring $\Delta_2$ does not satisfy (i\p ), (ii\p), (iii\p), and (v\p),
\item{} The coloring $\Delta_3$ does not satisfy (i\p), (ii\p), (iii\p), and (iv\p) for $k\geq 4$.
\end{itemize}
\end{lemma}

\begin{proof} Let $X=\{x_1<x_2<\cdots\}\subset {\mathbb N}$ be an infinite set, where by taking an infinite subset, we can assume that distances between consecutive elements of $X$  strictly increase and in particular that $X$ is rainbow under $\Delta_1$, i.e.,  $\Delta_1(x_1)<\Delta_1(x_2)<\cdots$.  
As a consequence, $X$ is rainbow under $\Delta_3$.
By Lemma \ref{lemma,2,3,4}, we can further assume that in the second coordinate the colors of $x_i$'s under $\Delta_2$ are all the same, thus for some $x'' \in \{1,2, \ldots p-1\}$, we have $\Delta_2(x_i) = (x_i', x'')$, for all $i$, and  $x_{i,j}' = x_{k, \ell}'$ if and only if $i = \ell$,  for $j>i$ and $k>\ell$.

Moreover, by Ramsey's theorem~\cite{ramsey} for pairs we can assume that 
one of the next two cases holds:\\

(a) $x_1'<x_2'< \cdots$, or \\
(b) for some $x'$, we have  $x_i'=x'$, for all $i$.\\

Next we consider the colorings $\Delta_1, \Delta_2, $ and $\Delta_3$ and patterns  from (i\p)-(v\p) that are not satisfied by these colorings.
\begin{itemize}

\item{$\mathbf {\Delta_1, \Delta_3, (i'), (iii\p):}$} ~~~ For  $\Delta_1$ patterns~(i\p) and~(iv\p) do not hold for $k \geq 2$. In particular, 
there are some sets $I, J\in \binom{{\mathbb N}}{k}$ such that $\min I = \min J$ and 
$\Delta_1(\sum^*_{j\in I} x_j)\neq  \Delta_1(\sum^*_{j\in J} x_j)$. 
 Thus, in particular the patterns (i\p) and (iv\p) do not hold for  $\Delta_3$.  \\

\item{$\mathbf {\Delta_1,  (iv\p),  k\geq 4:}$} ~~ For $\Delta_1$ pattern~(v\p) doesn't hold for $k \geq 4$. Indeed, 
consider  any $(k+2)$ positive integers 
$
x_{i_1} <x_{i_2} < x_{i_3} < \cdots  < \cdots < x_{i_{k}} < x_{i_{k+1}} < x_{i_{k+2}}$ in $X$.
Let $R, B,$ and $G$ be $k$-element sets of indices defined as follows:
\begin{eqnarray*}
R & = & \{i_1, i_3, i_4,\ldots,i_{k-1},  i_{k}, i_{k+2}\},\\
B & = &  \{i_1, i_3, i_4,\ldots, i_{k-1}, i_{k+1}, i_{k+2}\},\\
G & = & \{i_2, i_3, i_4, \ldots, i_{k-1}, i_{k}, i_{k+2}\}.
\end{eqnarray*}
Let $r= \sum^* _{j\in R} x_{j}$, $b= \sum^* _{j\in B} x_{j}$, and $g= \sum^* _{j\in G} x_{j}$.
With our assumption on strictly increasing differences between consecutive pairs, i.e., $x_{i_{k+1}}- x_{i_{k}} > x_{i_2} - x_{i_1}$  we have $b < g$, and with $x_{i_1} < x_{i_2}$ we have $g < r$, thus $b < g < r$. If pattern~(v)  were to hold, then we would have $\Delta_1(b) = \Delta_1(r)$, as  $\max B= \max R$ and $\min B = \min R$. By choice of the coloring we conclude $\Delta_1(g) = \Delta_1(b)$, which is a contradiction, as  $\min G \neq \min B$. So pattern~(v\p) does not hold.\\

\item{$\mathbf {\Delta_1,  (ii'), (v'),  k=2:}$} ~~{ For $\Delta_1$ patterns~(ii\p) and~(v\p) do not hold for $k = 2$}. Indeed, 
consider  for any $m\geq 4$ positive integers 
$x_{i_1} <x_{i_2}  < \cdots < x_{i_{m}}$
in $X$ sufficiently apart from each other. 
Consider the differences 
$
x_{i_{m}} -x_{i_1}, x_{i_{m}} -x_{i_2}, \cdots , x_{i_{m}} -x_{i_{m-1}}.
$
We can assume that  $x_{i_{m}}\in [ 2^\ell, 2^{\ell + 1} - 1]$ for some   $\ell$ and  $x_{i_j}< 2^{\ell -1}$, for $j=1,2, \ldots, m-1$.
Thus $x_{i_m} - x_{i_j} \in [2^{\ell-1}, 2^{\ell + 1} - 1]$, for $j=1,2, \ldots, m-1$, and hence at least two of  $x_{i_m} - x_{i_j}$, $j=1,2, \ldots, m-1$ are colored the same under $\Delta_1$, say $\Delta_1( x_{i_m} - x_{i_p}) = \Delta_1( x_{i_m} - x_{i_q})$. Clearly, the mimima
of the sets $\{i_p, i_m\}$ and $\{i_q, i_m\}$ are distinct.
 This  implies that patterns~(ii\p) and~(v\p) do not hold for $k=2$.\\

\item{$\mathbf {\Delta_1,  (ii'), k\geq 4:}$} ~~~~For $\Delta_1$ pattern~(ii\p) doesn't hold for {$k\geq 4$}.
Let $Q, S,$ and $T$ be $k$-element sets of indices defined as follows:
\begin{eqnarray*}
Q& = & \{i_1, i_2, ~  i_5, i_6, \ldots, i_{k+1}, i_{k+2}\}\\
S & = &  \{i_1, i_3, ~  i_5, i_6,\ldots, i_{k+1}, i_{k+2}\}\\
T & = & \{i_1, i_4, ~ i_5, i_6, \ldots, i_{k+1}, i_{k+2}\}.
\end{eqnarray*}
Let 
\begin{equation} \label{qs}
q= \sum^* _{j\in Q} x_{j},  ~~ s= \sum^* _{j\in S} x_{j},  \mbox{ and } t= \sum^* _{j\in T} x_{j}.
\end{equation}

 By considering indices $i_1, i_2,\ldots, i_{k+2}$ sufficiently far from each other, we can assume that  $x_{i_{k+2}}\geq 2^\ell$ for some  large $\ell$ and  $x_{i_j}< {2^{\ell -1}}$, for $j=1, \ldots, k+1$.
Thus $q, s, t \geq 2^{\ell-1}$ and $|s-t|, |s-q|, |t-q|< 2^{\ell-1}$. If $\min\{q, s, t\} \in [2^{\ell'}, 2^{\ell'+1}-1]$ for $\ell'\geq \ell-1$, 
then $q, t, s \in [2^{\ell'}, 2^{\ell'+1}-1+2^{\ell-1}]$, i.e, $\Delta(q), \Delta(t), \Delta(s) \in \{\ell', \ell'+1\}$. Thus two of $q,s,t$ have the same color and therefore $X_{alt}^*$ is not rainbow.  Assume without loss of generality that 
\begin{equation}\label{qs-Delta}
\Delta_1(q)=\Delta_1(s).
\end{equation} 

This in particular implies that pattern~(ii\p) does not hold.\\

 \item{$\mathbf { \Delta_2, (a), (ii'), (v'):} $} ~~~~If (a) is satisfied, then patterns~(ii\p) and~(v\p)  do not  hold for $\Delta_2$ for $k \geq 2$.
 Namely, we have for $s$ and $q$ defined in (\ref{qs}), that
 \begin{eqnarray*}
  \Delta_2( q) &= & \Delta_2( x_{i_{k+2}} - x_{i_{k+1}} + \cdots + x_{i_2} - x_{i_1}) = (x_{i_1}', -x'')=\\ 
 &=& \Delta_2( x_{i_{k+2}} - x_{i_{k+1}} +\cdots + x_{i_3} - x_{i_1})= \Delta_2(s). 
 \end{eqnarray*}
 so pattern~(ii\p) does not hold.
 
 Moreover, we have 
  \begin{eqnarray*}
   && \Delta_2( x_{i_{k+3}} - x_{i_{k+1}} + x_{i_k} - x_{i_{k-1}} +\cdots + x_{i_2} - x_{i_1}) = (x_{i_1}', -x'')=\\ 
 &=& \Delta_2( x_{i_{k+2}} - x_{i_{k+1}} +x_{i_k} - x_{i_{k-1}} +\cdots + x_{i_3} - x_{i_1})= \Delta_2(s). 
 \end{eqnarray*}
 so pattern~(v\p) does not hold.\\

\item{$\mathbf { \Delta_2, (a), (i'), (iii'):} $} If~(a) is satisfied, i.e., $x_i' < x_{i+1}'$, for all $i$, then patterns~(i\p) and~(iii\p) do not hold for $\Delta_2$ for  $k \geq 2$. Namely, we have
\begin{eqnarray*}
&& \Delta_2( x_{k+1} - x_{k} + \cdots + x_5 - x_4 + x_3 - x_1) \\
&=&
 \Delta_2( p^{x_{k}'} \cdot (p^{x_{k+1}' - x_{k}'} x_{k+1}'' - x_{k}'') + \cdots  + p^{x_1'} \cdot (p^{x_3' - x_1'} x_3'' - x_1'')) \\
 &=& (x_1', {p-x'')} \\
&\neq& \Delta_2( x_{k+1} - x_{k} + \cdots + x_5 - x_4 + x_3 - x_2) \\
&=&
 \Delta_2( p^{x_{k}'} \cdot (p^{x_{k+1}' - x_{k}'} x_{k+1}'' - x_{k}'') + \cdots  + p^{x_2'} \cdot (p^{x_3' - x_2'} x_3'' - x_2'')) \\
   &=& (x_2',p-x'').
 \end{eqnarray*}
\\

\item{$\mathbf {\Delta_2,  \Delta_3, (b), (ii'), k \geq 4:} $} ~~~~If  ~(b) is satisfied,  then pattern~(ii\p)  does not hold for $\Delta_2$ and $\Delta_3$,  $k\geq 4$.\\
In this case we have  $x_{i}' =x'$, for all $i$.
Recall that we assumed by Lemma~\ref{lemma,2,3,4}  that $X$ satisfies:\\
$x_{k,i}' = x_{j,i}'$, for all $1 \leq i < j < k$,   \\
$x_{j,i}'' \equiv x'' \bmod p$ for some $x'' \in \{1,2,\ldots, p-1\}$,
 for all $i < j$, and  \\
$x_{j,i}' < x_{\ell, k}'$ for all $i < j < k < \ell$.\\

For $k\geq 4$ consider $q$ and $s$ as in (\ref{qs}). 
We have from   (\ref{qs-Delta}) that $\Delta_1(q)=\Delta_1(s)$. We have 
\begin{eqnarray*}
\Delta_2(q) &=& \Delta_2(x_{i_{k+2}} - x_{i_{k+1}}  + \cdots + x_{i_{6}} -  x_{i_{5}}   +     x_{i_{2}} 
 - x_{i_{1}})\\
&=&  \Delta_2( p^{x'} \cdot ( p^{x_{i_{k+2},i_{k+1}}'}x_{i_{k+2},i_{k+1}}'' + \cdots + p^{x_{i_{6},i_{5}}'}x_{i_{6},i_{5}}'' +   p^{x_{i_{2},i_{1}}'}x_{i_{2},i_{1}}''))
\\
&=&
(x' + x_{i_{2},i_{1}}', -x''),  \mbox{~~ and}
\end{eqnarray*}
\begin{eqnarray*}
\Delta_2(s) 
&=&\Delta_2(x_{i_{k+2}} - x_{i_{k+1}}  +  \cdots + x_{i_{6}} -  x_{i_{5}}  +     x_{i_{3}} - 
 x_{i_{1}})\\
&=&  \Delta_2( p^{x'} \cdot ( p^{x_{i_{k+2},i_{k+1}}'}x_{i_{k+2},i_{k+1}}'' + \cdots +   p^{x_{i_{6},i_{5}}'}x_{i_{6},i_{5}}''+  p^{x_{i_{3},i_{1}}'}x_{i_{3},i_{1}}''))\\
&=&
(x' + x_{i_{2},i_{1}}', -x''),
\end{eqnarray*}
thus $\Delta_2(q) = \Delta_2(s)$.
As $ x_{i_{2}, i_1}'=  x_{i_{3},i_{1}}'$ and $x_{\ell, k}' > x_{j,i}'$ for all $1 \leq i < j < k < \ell$, we have
$\Delta_3(q) = \Delta_3(s)$.   Thus pattern (iii\p) is not satisfied by $\Delta_2$ and thus by $\Delta_3$.\\

\item{$\mathbf{\Delta_2, (b), (i'), (iii'), (v')}$} Patterns~(i\p) and~(iii\p) do not hold under $\Delta_2$ for $k \geq 4$ when (b) is satisfied.
Consider
\begin{eqnarray*}
 && \Delta_2(x_{i_{k+2}} - x_{i_{k+1}}   +  \cdots +    x_{i_{6}} - 
 x_{i_{5}}  +    x_{i_{4}} - 
 x_{i_{3}}  +  x_{i_{2}} - 
 x_{i_{1}})\\
&=&  \Delta_2( p^{x'} \cdot ( p^{x_{i_{k+2},i_{k+1}}'}x_{i_{k+2},i_{k+1}}'' +  \cdots + p^{x_{i_{6},i_{5}}'}x_{i_{6},i_{5}}'' +  p^{x_{i_{4},i_{3}}'}x_{i_{4},i_{3}}''+  p^{x_{i_{2},i_{1}}'}x_{i_{2},i_{1}}''))
\\
&=&
(x' + x_{i_{2},i_{1}}', z'')
  \\
&\neq & \Delta_2(x_{i_{k+2}} - x_{i_{k+1}}  +  \cdots    +     x_{i_{6}} - 
 x_{i_{5}}  +     x_{i_{4}} - 
 x_{i_{3}})\\
&=&  \Delta_2( p^{x'} \cdot ( p^{x_{i_{k+2},i_{k+1}}'}x_{i_{k+2},i_{k+1}}'' + \cdots +   p^{x_{i_{6},i_{5}}'}x_{i_{6},i_{5}}''+  p^{x_{i_{4},i_{3}}'}x_{i_{4},i_{3}}'')) \\
&=&
(x' + x_{i_{4},i_{3}}', z''),
\end{eqnarray*}
as $ x_{i_{2},i_{1}}' <  x_{i_{4},i_{3}}'$,
so patterns~(ii\p) and~(iii\p) do not hold under $\Delta_2$. \\

Pattern~(v\p) does not hold for $\Delta_2$ for $k\geq 4$. Namely, consider 
\begin{eqnarray*}
 && \Delta_2 (x_{i_k} -x_{i_{k-1}} + \cdots + x_{i_{4}} -  x_{i_{3}}   +     x_{i_{2}} 
 - x_{i_{1}})\\
&=&  \Delta_2( p^{x'} \cdot ( p^{x_{i_{k},i_{k-1}}'}x_{i_{k},i_{k-1}}'' + \cdots + p^{x_{i_{4},i_{3}}'}x_{i_{4},i_{3}}'' +   p^{x_{i_{2},i_{1}}'}x_{i_{2},i_{1}}''))
\\
&=&
(x' + x_{i_{2},i_{1}}', z'') \\
&=& 
\Delta_2(x_{i_{k+1}} - x_{i_{k-1}}  + x_{i_{k-2}} - x_{i_{k-3}}+ \cdots + x_{i_{4}} -  x_{i_{3}}  +     x_{i_{2}} - 
 x_{i_{1}})\\
&=&  \Delta_2( p^{x'} \cdot ( p^{x_{i_{k+1},i_{k-1}}'}x_{i_{k+1},i_{k-1}}'' + \cdots +   p^{x_{i_{4},i_{3}}'}x_{i_{4},i_{3}}''+  p^{x_{i_{2},i_{1}}'}x_{i_{2},i_{1}}''))\\
&=&
(x' + x_{i_{2},i_{1}}', z''),
\end{eqnarray*}
as $x_{j,i}' < x_{\ell, k}'$, for all $ i < j < k < \ell$, 
so pattern~(v\p) does not hold under $\Delta_2$.\\

{Let $k=2$. Consider $x_{i_2} - x_{i_1}$ and $x_{i_3} - x_{i_1}$. By Lemma~\ref{lemma,2,3,4} we have
$\Delta_2(x_{i_2} - x_{i_1}) = \Delta_2(x_{i_3} - x_{i_1})$, as $x_{i_2,i_1}' =x_{i_3, i_1}'$, so patterns~(ii\p) and~(iii\p)  do not hold for $k=2$ under $\Delta_2$. Moreover, consider $x_{i_2} - x_{i_1}$ and $x_{i_4} - x_{i_3}$. By Lemma~\ref{lemma,2,3,4} we have
$\Delta_2(x_{i_2} - x_{i_1}) \neq \Delta_2(x_{i_4} - x_{i_3})$, as $x_{i_2,i_1}' <x_{i_4, i_1}'$, so pattern~(i\p) also does not hold under $\Delta_2$ for $k=2$.}\\

\item{$\mathbf { \Delta_3, (b), (iii'), k\geq 4:} $} ~~~~ If (b) is satisfied, then pattern~(iii\p) does not hold for  $\Delta_3$ and $k\geq 4$. Indeed:

\begin{eqnarray*}
 && \Delta_2(x_{i_{k+2}} - x_{i_{k+1}}  +  \cdots    +     x_{i_{6}} - 
 x_{i_{5}}  +     x_{i_{2}} - 
 x_{i_{1}})\\
&=&  \Delta_2( p^{x'} \cdot ( p^{x_{i_{k+2},i_{k+1}}'}x_{i_{k+2},i_{k+1}}'' +  \cdots +   p^{x_{i_{6},i_{5}}'}x_{i_{6},i_{5}}''+  p^{x_{i_{2},i_{1}}'}x_{i_{2},i_{1}}''))
\\
&=&
(x' + x_{i_{2},i_{1}}', z'')
  \\
&\neq & \Delta_2(x_{i_{k+2}} - x_{i_{k+1}}  +  \cdots +     x_{i_{6}} - 
 x_{i_{5}}  +     x_{i_{4}} - 
 x_{i_{3}})\\
&=&  \Delta_2( p^{x'} \cdot ( p^{x_{i_{k+2},i_{k+1}}'}x_{i_{k+2},i_{k+1}}'' + \cdots +   p^{x_{i_{6},i_{5}}'}x_{i_{6},i_{5}}''+  p^{x_{i_{4},i_{3}}'}x_{i_{4},i_{3}}'')) \\
&=&
(x' + x_{i_{4},i_{3}}', z''),
\end{eqnarray*}
 as $ x_{i_{2},i_{1}}' <  x_{i_{4},i_{3}}'$.\\

\item{ $\mathbf { \Delta_3,  (a), (ii') }$}  ~~~~  If (a) is satisfied then pattern (ii\p) does not hold for $\Delta_3$.  Indeed, considering the alternating sums $s$ and $q$ defined in (\ref{qs}), we have $\Delta_1(s)=\Delta_1(q)$ and $\Delta_2(s)=\Delta_2(q)$, thus $\Delta_3(s)=\Delta_3(q)$.  Hence $X_{alt}^*$ is not rainbow and pattern~(ii\p) does not hold.
\end{itemize}

Summarising all these items, we see that the lemma holds.%(i\p)
\end{proof}

\begin{proof}[Proof of Theorem \ref{prop-1,+1}]

Let $\Delta\colon {\mathbb N}  \longrightarrow {\mathbb N}$  be an arbitrary coloring.   Consider another
coloring $\Delta'\colon \PP'  \longrightarrow {\mathbb N},$
defined  for any finite set $S\subseteq \mathbb{N}$ by 
$$
\Delta' ( S) = \Delta \left( \sum_{i\in S }  2^i \right).
$$

By Theorem~\ref{theo.a.canonicaln} applied to a family ${\mathcal A}$, whose members are single element subsets of $\mathbb{N}$, there exists a subfamily ${\mathcal B} =\{B_1 < B_2 < \cdots \}$ of ${\mathcal A}'$,  a family of  finite subsets of $\mathbb{N}$, such that one of the five canonical patterns~(i)-(v)  with respect to the coloring $\Delta'$ hold.  

Consider the integers
$$x_j = \sum_{t\in B_1 \cup B_2 \cup \cdots \cup B_j} 2^t, \,\, j=1,2,\ldots.$$ 
Let $X=\{x_1 < x_2 < \cdots \} \subseteq {\mathbb N}$ be an infinite set.
For a set $J= \{j_1 < j_2 < \cdots < j_k\}$, even $k$, of positive integers,   let 
$$J_{alt} = \bigcup _{\ell =1}^{k/2}\bigcup_{{s}=1}^{j_{2\ell} - j_{2\ell-1}} \{{j_{2\ell-1} + {s}}\},  \mbox{    i.e.,}$$ 
$$J_{alt} = \{j_1+1, j_1 + 2, \ldots, j_2, ~\,~\,~\, j_3+1, j_3 +2, \ldots, j_4, ~\,~\,~\ldots,~~\,\,~ j_{k-1}+1, j_{k-1} + 2, \ldots, j_k\}$$
and let
$$B_{J_{alt}} = \bigcup_{i\in J_{alt}} B_i.$$

Then we have  
$$\sum^*_{j\in J} x_j = \sum_{ t \in  B_{J_{alt}}}  2^t.$$

In particular, we see that $\Delta$ on $X\cup X_{alt}^*$ corresponds to $\Delta'$ on specific finite unions of sets $B_i$'s:
$$\Delta'(B_1\cup B_2 \cup \cdots \cup B_j) = \Delta(x_j) ~  \mbox{   and   } ~ \Delta'( \bigcup_{i\in J_{alt}} B_i) =\Delta(\sum^*_{j\in J} x_j ).$$

\vskip 1cm

First  we shall show that the sets $X= \{ x_1 < x_2  < \ldots \}$ and $X_{alt}^*$ satisfy  the five canonical patterns described in the statement of the theorem.
\begin{itemize}
\item{} If we have pattern~(i)  of Theorem~\ref{theo.a.canonicaln} for coloring $\Delta'$, i.e., the set of all finite unions of $B_i$'s is monochromatic, then 
 $X \cup X_{alt}^*$ is  monochromatic under $\Delta$.   Similarly, if we have pattern~(ii) of Theorem~\ref{theo.a.canonicaln} for $\Delta'$,  
then  $X \cup X_{alt}^*$ is  rainbow under $\Delta$.
\item{}If we have pattern~(iii) of Theorem~\ref{theo.a.canonicaln} for $\Delta'$, then  $X$ is rainbow under $\Delta$, 
and $\Delta (x_\ell) = \Delta( \sum^*_{j\in J} x_j)$
 if and only if $\ell = \max J$, and $\Delta( \sum^*_{i\in I} x_i)= \Delta( \sum^*_{j\in J} x_j)$ if and only if $\max I= \max J$, for all $\ell\in \mathbb{N}$ and $I,J\in \binom{\mathbb{N}}{k}$.
\item{} If we have pattern~(iv) of Theorem~\ref{theo.a.canonicaln} for  $\Delta'$, then  $X$ is monochromatic under $\Delta$, and
  $\Delta (x_\ell) \neq \Delta( \sum^*_{j\in J} x_j)$, and  
$\Delta( \sum^*_{i\in I} x_i)= \Delta( \sum^*_{j\in J} x_j)$ if and only if $\min I= \min J$, for all $\ell\in \mathbb{N}$ and $I,J\in \binom{\mathbb{N}}{k}$.
\item{}If we have pattern~(v)  of Theorem~\ref{theo.a.canonicaln} for  $\Delta'$, then $X$ is rainbow under $\Delta$,  and $\Delta (x_\ell) \neq \Delta( \sum^*_{j\in J} x_j)$, and
 ~~$\Delta( \sum^*_{i\in I} x_i)$ $= \Delta( \sum^*_{j\in J} x_j)$ if and only if $\min I=\min J$ and $\max I =\max J$, for all $\ell\in \mathbb{N}$ and $I, J\in \binom{\mathbb{N}}{k}$.
\end{itemize}
Thus one of the patterns~(i)-(v) holds for $\Delta$.

Clearly patterns~(i) and~(ii) are unavoidable by considering a monochromatic and a rainbow coloring of $\mathbb{N}$, respectively. The necessity of each of the patterns (iii)-(v) follows from Lemma \ref{unavoidable-i-v} since if the pattern (i\p), (ii\p), (iii\p), (iv\p),  or (v\p) doesn't hold for a coloring, then the respective pattern (i),  (ii), (iii), (iv),  or (v) also doesn't hold.

This completes the proof of Theorem~\ref{prop-1,+1}.
\end{proof}

\vskip 1cm

%%%%%%%%%%%%%%%%%%%%%%%%%%%%%%%%%%%%%%%%%%%%%%%%%%%
%%%%%%%%%%%%%%%%%%%%%%%%%%%%%%%%%%%%%%%%%%%%%%%%%%%
%%%%%%%%%%%%%%%%%%%%%%%%%%%%%%%%%%%%%%%%%%%%%%%%%%%
\section{The Set $X \cup \{ x_{k} - x_{k-1} + \cdots + x_{3} - x_{2} + x_{1}: ~ x_1, x_2, \ldots, x_k \in X,   ~ x_{1}< x_{2} < \cdots < x_{k} \}$, $k\geq 3$~odd}\label{X-alternating-odd}
%%%%%%%%%%%%%%%%%%%%%%%%%%%%%%%%%%%%%%%%%%%%%%%%%%%
%%%%%%%%%%%%%%%%%%%%%%%%%%%%%%%%%%%%%%%%%%%%%%%%%%%
%%%%%%%%%%%%%%%%%%%%%%%%%%%%%%%%%%%%%%%%%%%%%%%%%%%

We use the following notation for a $k$-element set $J=\{j_1<j_2< \cdots <j_k\}$ with $k$ odd and positive integers $x_j$, $j\in J$:
$$\sum^{**}_{j\in J} x_j =   x_{j_k} - x_{j_{k-1}} +x_{j_{k-2}} - x_{j_{k-3}} + \cdots + x_{j_3} - x_{j_2} + x_{j_1}.$$
We refer to this also as \emph{alternating sum}. 

\begin{theorem} \label{prop+1,-1}
 Let $k \geq 3$ be an odd integer. For any coloring $\Delta \colon {\mathbb N} \longrightarrow {\mathbb N}$ there exists an infinite set $X = \{x_1 < x_2 <\cdots\}$ of positive integers 
such that the sets
$X$ and $X_{alt}^{**}= X_{alt}^{**}(k)=\{  \sum^{**}_{j\in J}x_j :  J\in \binom{\mathbb{N}}{k}\}$ are  colored according to one of the following patterns:
\begin{itemize} 
\item[(i)] $X \cup X_{alt}^{**}$ is monochromatic, or
\item[(ii)]  $X \cup X_{alt}^{**}$  is rainbow, or 
\item[(iii)]$X$  is rainbow  and  $\Delta(\sum^{**}_{j\in J } x_j)= \Delta(x_{\max J})$, for all $J\in \binom{\mathbb N}{k}$.
\end{itemize}
None of these three patterns may be omitted without violating the theorem.
\end{theorem}

In view of the remarks directly below Theorem~\ref{ghl} we have the following consequence of Theorem~\ref{prop+1,-1}.

\begin{corollary} \label{cor+1,-1}
Let $k,r\geq 2$ be integers, where $k\geq 3$ is odd. Let $\Delta \colon {\mathbb N} \longrightarrow [r]$ be a
coloring. 
Then, one can find an infinite set X= \{$x_1 < x_2 < \cdots\}$ of positive integers, such that the set
$X \cup X_{alt}^{**}(k)$ is  monochromatic.
\end{corollary}

\begin{proof}[Proof of Theorem \ref{prop+1,-1}]
The proof is similar to the proof of Theorem~\ref{prop-1,+1}. Let $\Delta\colon {\mathbb N}  \longrightarrow {\mathbb N}$
 be an arbitrary coloring.

 Consider another coloring $\Delta'\colon \PP'  \longrightarrow {\mathbb N},$
defined  for any finite set $S\subseteq \mathbb{N}$ by 
$$
\Delta' ( S) = \Delta \left( \sum_{i\in S }  2^i \right).
$$

  By Theorem~\ref{theo.a.canonicaln} applied to a family ${\mathcal A}$, whose members all are single element subsets of $\mathbb{N}$, there exists a subfamily ${\mathcal B} =\{B_1 < B_2 < \cdots \}$ of ${\mathcal A}'$,  a family of  finite subsets of $\mathbb{N}$, such that one of the five canonical patterns~(i)-(v)  with respect to the coloring $\Delta'$ holds.  

Consider the integers
$$x_j = \sum_{t\in B_1 \cup B_2 \cup \cdots \cup B_j} 2^t, \,\, j=1,2,\ldots.$$ 
Let $X=\{x_1 < x_2 < \cdots \} \subseteq {\mathbb N}$ be an infinite set.
For a set $J= \{j_1 < j_2 < \cdots < j_k\}$, odd $k$, of positive integers,   let 

$$J_{alt} = \{1, 2, \ldots, j_1\}  \bigcup _{\ell =1}^{(k-1)/2}\bigcup_{{s}=1}^{j_{2\ell+1} - j_{2\ell}} \{{j_{2\ell} + {s}}\}, \mbox{  i.e., }$$
$$J_{alt} = \{ 1, 2, \ldots, j_1,   ~\ ~\ ~\, j_2+1, j_2 + 2, \ldots, j_3, ~\,~\,~\, j_4+1, j_4 +2, \ldots, j_5, ~\,~\,~\ldots,~~\,\,~ j_{k-1}+1, j_{k-1} + 2, \ldots, j_k\}$$
and let
$$B_{J_{alt}} = \bigcup_{i\in J_{alt}} B_i.$$

Then we have  
$$\sum^{**}_{j\in J} x_j = \sum_{ t \in  B_{J_{alt}}}  2^t.$$

In particular, we see that $\Delta$ on $X\cup X_{alt}^*$ corresponds to $\Delta'$ on specific finite unions of sets $B_i$'s:
$$\Delta'(B_1\cup B_2 \cup \cdots \cup B_j) = \Delta(x_j) ~  \mbox{   and   } ~ \Delta'( \bigcup_{i\in J_{alt}} B_i) =\Delta(\sum^{**}_{j\in J} x_j ).$$

First  we shall show that $X$ and $X_{alt}^{**}$ satisfy  the three canonical patterns described in the statement of the theorem.
\begin{itemize}
\item{} If we have pattern~(i)  of Theorem~\ref{theo.a.canonicaln} for $\Delta'$, i.e., the set of all finite unions of $B_i$'s is monochromatic, then 
 $X \cup X_{alt}^{**}$ is  monochromatic under $\Delta$.   Similarly, if we have pattern~(ii) of Theorem~\ref{theo.a.canonicaln} for $\Delta'$,  
then  $X \cup X_{alt}^{**}$ is  rainbow under $\Delta$.
\item{}If we have pattern~(iii) of Theorem~\ref{theo.a.canonicaln} for  $\Delta'$, then  $X$ is rainbow under $\Delta$, and
$ \Delta( \sum^{**}_{j\in J} x_j)=\Delta (x_{\max J}) $
for all $J\in \binom{\mathbb{N}}{k}$.
\item{} If we have pattern~(iv) of Theorem~\ref{theo.a.canonicaln} for  $\Delta'$, then  $X\cup X_{alt}^{**}$ is monochromatic, as $\min B_1= \min \cup_{\ell = 1}^{j_1} B_\ell \cup_{i=j_{k-1} + 1}^{j_{k}} B_{i}= \min \cup_{i=1}^{j} B_i$,
 for all $j \in {\mathbb N}$   and $J
 = \{j_1  < j_2 < \cdots < j_k \} \in \binom{\mathbb{N}}{k}$, so pattern~(iv) coincides with pattern~(i).
\item{}If we have pattern~(v) of Theorem~\ref{theo.a.canonicaln} for $\Delta'$,  then $X$ is rainbow under $\Delta$, and
$ \Delta( \sum^{**}_{j\in J} x_j)=\Delta (x_{\max J}) $
for all $J\in \binom{\mathbb{N}}{k}$,  so pattern~(v) coincides with pattern~(iii).
\end{itemize}
Thus one of the patterns~(i)-(iii) holds for $\Delta$.\\

Patterns~(i) and~(ii) are clearly necessary. 

Consider the coloring   $\Delta_1 \colon {\mathbb N} \longrightarrow {\mathbb N}$, where 
$\Delta_1(x) = i$ if and only if $x \in [2^i, 2^{i+1} - 1]$, for a non-negative integer $i$.
We recall some properties of $\Delta_1$:
\begin{itemize}
\item{} For  $\Delta_1$ pattern~(i) does not hold for $k \geq 3$.\\

\item{} We shall show next that for $\Delta_1$ pattern~(ii) does not hold for $k \geq 3$.
Let $Q, S,$ and $T$ be $k$-element sets of indices defined as follows:
\begin{eqnarray*}
Q& = & \{i_1, i_2, ~  i_5,i_6,  \ldots, i_{k+1}, i_{k+2}\}\\
S & = &  \{i_1, i_3, ~  i_5, i_6, \ldots, i_{k+1}, i_{k+2}\}\\
T & = & \{i_1, i_4, ~ i_5, i_6,\ldots, i_{k+1}, i_{k+2}\}.
\end{eqnarray*}
Let $q= \sum^{**} _{j\in Q} x_{j}$, $s= \sum^{**} _{j\in S} x_{j}$, and $t= \sum^{**} _{j\in T} x_{j}$.
 By considering indices $i_1, i_2, \ldots, i_{k+2}$ sufficiently far from each other, we can assume that  $x_{i_{k+2}}\geq 2^\ell$ for some  large $\ell$ and  $x_{i_j}< 2^{{\ell -1}}$, for $j=1,2, \ldots, k+1$.
Thus $q, s, t \geq 2^{\ell-1}$ and $|s-t|, |s-q|, |t-q|< 2^{\ell-1}$. If $\min\{q, s, t\} \in [2^{\ell'}, 2^{\ell'+1}-1]$ for some $\ell'\geq \ell-1$, 
then $q, t, s \in [2^{\ell'}, 2^{\ell'+1}-1+2^{\ell-1}]$, i.e, $\Delta_1(q), \Delta_1(t), \Delta_1(s) \in \{\ell', \ell'+1\}$. Thus two of $q,s,t$ have the same color and therefore $X_{alt}^{**}$ is not rainbow.   This in particular implies that pattern~(ii) does not hold for $\Delta_1$.
\end{itemize}
Therefore, pattern~(iii) can not be omitted for $k \geq 3$ without violating the theorem.
\end{proof}

%%%%%%%%%%%%%%%%%%%%%%%%%%%%%%%%%%%%%%%%%%%%%%%%%%%%%

\vskip 1cm

%\vskip 1cm
%%%%%%%%%%%%%%%%%%%%%%%%%%%%%%%%%%%%%%%%%%%%%%%%%%%%%%%%%%%%%
%%%%%%%%%%%%%%%%%%%%%%%%%%%%%%%%%%%%%%%%%%%%%%%%%%%%%%%%%%%%%
%%%%%%%%%%%%%%%%%%%%%%%%%%%%%%%%%%%%%%%%%%%%%%%%%%%%%%%%%%%%%
%%%%%%%%%%%%%%%%%%%%%%%%%%%%%%%%%%%%%%%%%%%%%%%%%%%%%%%%%%%%%
%%
\section{The Set $\{x_1+x_2 + \cdots +x_k:  ~~  x_1, x_2, \ldots, x_k \in X, ~ x_1<x_2 < \cdots<x_k\}$}\label{x1+x2+xk}
%%
%%%%%%%%%%%%%%%%%%%%%%%%%%%%%%%%%%%%%%%%%%%%%%%%%%%%%%%%%%%%%
%%%%%%%%%%%%%%%%%%%%%%%%%%%%%%%%%%%%%%%%%%%%%%%%%%%%%%%%%%%%%
%%%%%%%%%%%%%%%%%%%%%%%%%%%%%%%%%%%%%%%%%%%%%%%%%%%%%%%%%%%%%
If for a fixed $k$, we only consider $k$-term sums of integers, as a  consequence of Theorem~\ref{theo.a.canonical} we have the following, where only three canonical patterns show up.

\begin{theorem}\label{theo.xxx}
Let $k \geq 2$ be a fixed positive integer. For every coloring $\Delta\colon
{\mathbb N} \longrightarrow {\mathbb N}$ there exist infinitely 
many positive
integers $x_1 < x_2 < \cdots$, such that one of the following holds:
\begin{itemize} 
	\item[(i)] $\Delta (\sum_{i\in I } x_i) = \Delta (\sum_{j\in J } x_j)$, for all   $I, J \in \binom{\mathbb N}{k}$, or
	\item[ (ii)] $\Delta (\sum_{i\in I } x_i) = \Delta (\sum_{j\in J } x_j)$ if and only if $I=J$,
for all   $I, J  \in \binom{\mathbb N}{k}$, or
\item[ (iii)]  $\Delta (\sum_{i\in I } x_i) = \Delta (\sum_{j\in J } x_j)$ if and only if $\mbox{max } I = \mbox{max } J$,
for all   $I, J  \in \binom{\mathbb N}{k}$.
\end{itemize}
None of these three patterns may be omitted without violating the theorem.
\end {theorem}

\begin{proof}

 Let  $\Delta\colon {\mathbb N} \longrightarrow {\mathbb N}$ be given. 
 Let $\Delta'\colon {\mathbb N} \longrightarrow {\mathbb N}$ be defined as $\Delta'(x) = \Delta(kx)$, $x \in {\mathbb N}$.
 By Theorem~\ref{theo.a.canonical}  applied to $\Delta'$, there exist infinitely 
many positive integers $y_0 < y_1 < \cdots$ such that one of the five patterns there applies. This implies that one of the five patterns applies to $ky_0< ky_1< \cdots $ under $\Delta$.  Let $z_i = ky_i$, $i=1, 2, \ldots$.

Set $x_i = z_0/k + z_i$, for all $i \in {\mathbb N}$. Then for a $k$-element set $S \subset{\mathbb N}$ we have  $\sum_{i \in S} x_i = z_ 0 + \sum_{i \in S} z_i$. If we have pattern~(i) of Theorem~\ref{theo.a.canonical}, then all these $(k+1)$-term sums have the same color. If we are in pattern~(ii) of Theorem~\ref{theo.a.canonical},
then all these sums are rainbow. If we have pattern~(iii) of Theorem~\ref{theo.a.canonical},
then all these sums are maximum-colored. If we have pattern~(iv) of Theorem~\ref{theo.a.canonical},
then, as  $\sum_{i \in S} x_i = z_ 0 + \sum_{i \in S} z_i$, all these sums  have the same color.
If we have pattern~(v) of Theorem~\ref{theo.a.canonical},
then again, as  $\sum_{i \in S} x_i = z_ 0 + \sum_{i \in S} z_i$, all these sums are   maximum-colored.
Thus  the integers $x_1<x_2< \ldots $ satisfy one of the patterns (i)-(iii).\\

The necessity  of the  three patterns in Theorem~\ref{theo.xxx} 
follows as before. Patterns~(i) and (ii) must be there. To  the last pattern~(iii) one applies the coloring $\Delta\colon {\mathbb N}  \longrightarrow {\mathbb N}$, where $\Delta (x) = \ell$ if and only $x \in [k^{\ell -1}, k^{\ell } -1 ]$, for positive integers $\ell$. Let $X= \{ x_1 < x_2 < \cdots \}$ be an infinite set of positive integers. Clearly, we do not have  pattern~(i), as $X$ is infinite. Moreover, all $k$-term sums cannot be 
rainbow. Namely, 
by thinning, we can assume that all $k$-term sums are pairwise distinct.  If $x_m \in [k^{\ell -1}, k^{\ell } -1 ]$ and $m \geq {\max \{4,k+1\}}$, then for all $k$-term sums $\sum_{i \in I} x_i$ with  $\mbox{max } I = m$, it is
$ k^{\ell -1} < \sum_{i \in I} x_i < k^{\ell + 1} - 1$, so each such sum has one out of at most two possible colors, but there are $\binom{m-1}{k-1} \geq 3$ possible sums. This shows that all  three patterns are needed.
\end{proof}

\vskip 1cm
%%%%%%%%%%%%%%%%%%%%%%%%%%%%%%%%%%%%%%%%%%%%%%%%%%%%%%%%%%%%%%%%%%%%%%%%%
%%%%%%%%%%%%%%%%%%%%%%%%%%%%%%%%%%%%%%%%%%%%%%%%%%%%%%%%%%%%%%%%%%%%%%%%%
\section{The set $\{ x_k-x_{k-1} + x_{k-2}- x_{k-3}+ \cdots +x_2-x_1: ~ x_1, x_2, \ldots, x_k \in X, ~ x_1<x_2< \cdots < x_k\}$, \\ $k$ even }\label{section_even,ell}
%%%%%%%%%%%%%%%%%%%%%%%%%%%%%%%%%%%%%%%%%%%%%%%%%%%%%%%%%%%%%%%%%%%%%%%%%
%%%%%%%%%%%%%%%%%%%%%%%%%%%%%%%%%%%%%%%%%%%%%%%%%%%%%%%%%%%%%%%%%%%%%%%%%

We recall the following notation of an alternating sum for a $k$-element set $J=\{j_1<j_2< \cdots <j_k\}$ and positive integers $x_j$, $j\in J$:
$$\sum^*_{j\in J} x_j =   x_{j_k} - x_{j_{k-1}} +x_{j_{k-2}} - x_{j_{k-3}} + \cdots + x_{j_2} - x_{j_1}.$$

\begin{theorem} \label{prop-1,+1,ell}
Let $k\geq 2$ be an even integer. Let $\Delta \colon {\mathbb N} \longrightarrow {\mathbb N}$ be an arbitrary 
coloring. 
Then, one can find an infinite set $X= \{ x_1 < x_2 < \cdots\}$ of positive integers, such that the set
$X_{alt}^*= X_{alt}^*(k)=\{  \sum^*_{j\in J}x_j :  J\in \binom{\mathbb{N}}{k}\}$ is  colored according to one of the following patterns.
\begin{itemize}
\item[(i)]  $X_{alt}^*$ is  monochromatic, or
\item[(ii)]  $X_{alt}^* $  is rainbow, or
\item[(iii)]   $\Delta(\sum^*_{i\in I } x_i)= \Delta(\sum^*_{j\in J} x_j)$ if and only if $\max I = \max J$, for all $I,J\in \binom{\mathbb N}{k}$, or 
\item[(iv)]   $\Delta( \sum^*_{i\in I} x_i)= \Delta( \sum^*_{j\in J} x_j)$ if and only if $\min I= \min J$, for all $I,J\in \binom{\mathbb{N}}{k}$,  or
 \item[(v)]   $\Delta(\sum^*_{i\in I } x_i)= \Delta (\sum^*_{j\in J} x_j)$ if and only if $\min I=\min J$ and $\max I=\max J$,  for all $I,J\in \binom{\mathbb N}{k}$.
\end{itemize}
{For $k \geq 4$ none of the five patterns may be omitted without violoating the theorem. For $k=2$ patterns~(ii) and~(v) coincide. In the latter case, none of the four patterns~(i)-(iv) may be omitted without violating the theorem.}
\end{theorem}

\begin{proof}
The unavoidability of one of the  five patterns follow from  Theorem~\ref{prop-1,+1}. Clearly  patterns~(i) and~(ii) are necessary by considering a monochromatic and a rainbow coloring of $\mathbb{N}$, respectively.\\
The necessity of the patterns (iii)-(v) follows from Lemma \ref{unavoidable-i-v}.  
\end{proof}

\vskip 1cm
%%%%%%%%%%%%%%%%%%%%%%%%%%%%%%%%%%%%%%%%%%%%%%%%%%%%%%
%%%%%%%%%%%%%%%%%%%%%%%%%%%%%%%%%%%%%%%%%%%%%%%%%%%%%%
\section{The Set $ \{ x_{k} - x_{k-1} + \cdots + x_{3} - x_{2} + x_{1}: x_1, x_2, \ldots, x_k \in X,   x_{1}< x_{2} < \cdots < x_{k} \}$, \\ $k\geq 3$ odd} \label{alternating-odd}
%%%%%%%%%%%%%%%%%%%%%%%%%%%%%%%%%%%%%%%%%%%%%%%%%%%%%%%
%%%%%%%%%%%%%%%%%%%%%%%%%%%%%%%%%%%%%%%%%%%%%%%%%%%%%%%

Recall the following notation for a $k$-element set $J=\{j_1<j_2< \cdots <j_k\}$ with $k$ odd and positive integers $x_j$, $j\in J$:
$$\sum^{**}_{j\in J} x_j =   x_{j_k} - x_{j_{k-1}} +x_{j_{k-2}} - x_{j_{k-3}} + \cdots + x_{j_3} - x_{j_2} + x_{j_1}.$$

\begin{theorem} \label{prop+1,-1,ell}
Let  $k \geq 3$ be an odd integer. For any coloring $\Delta \colon {\mathbb N} \longrightarrow {\mathbb N}$ there exists an infinite set $X=\{x_1 < x_2 <\cdots\}$ of positive integers
such that the set
$X_{alt}^{**}= X_{alt}^{**}(k)=\{  \sum^{**}_{j\in J}x_j :  J\in \binom{\mathbb{N}}{k}\}$ is  colored according to one of the following patterns: 
\begin{itemize} 
\item[(i)] $X_{alt}^{**}$ is monochromatic, or
\item[(ii)]  $ X_{alt}^{**}$  is rainbow, or 
\item[(iii)]  $\Delta(\sum^{**}_{i\in I } x_i)=\Delta(\sum^{**}_{j\in J } x_j)$ if and only if $\max I= \max J$,  for all $I, J\in \binom{\mathbb N}{k}$.
\end{itemize}
Moreover, none of the patterns could be omitted.
\end{theorem}

\begin{proof}
The unavoidability of one  of the patterns (i)-(iii) follows from Theorem \ref{prop+1,-1}.

Patterns~(i) and~(ii) in the theorem are clearly necessary.  Next we show that the pattern (iii) is necessary as well. 

Let $X = \{x_1 < x_2 < \cdots \}$ be an infinite set of positive integers.
Let $\Delta_1 \colon {\mathbb N} \longrightarrow {\mathbb N}$, where 
$\Delta_1(x) = i$ if and only if $x \in [2^{i-1}, 2^{i} - 1]$, for a positive integer $i$.
We recall some properties of $\Delta_1$:
\begin{itemize}
\item{} For  $\Delta_1$ pattern~(i) does not hold for $k\geq 3$.\\

\item{} For $\Delta_1$ pattern~(ii) does not hold for $k \geq 3$.
Let $Q, S,$ and $T$ be $k$-element sets of indices defined as follows:
\begin{eqnarray*}
Q& = & \{i_1, i_2, ~  i_5,  \ldots, i_{k+1}, i_{k+2}\}\\
S & = &  \{i_1, i_3, ~  i_5, \ldots, i_{k+1}, i_{k+2}\}\\
T & = & \{i_1, i_4, ~ i_5, \ldots, i_{k+1}, i_{k+2}\}.
\end{eqnarray*}
Let $q= \sum^{**} _{j\in Q} x_{j}$, $s= \sum^{**} _{j\in S} x_{j}$, and $t= \sum^{**} _{j\in T} x_{j}$.
 By considering indices $i_1, i_2,\ldots, i_{k+2}$ sufficiently far from each other, we can assume that  $x_{i_{k+2}}\geq 2^\ell$ for some  large $\ell$ and  $x_{i_j}< 2^{{\ell -1}}$, for $j=1,2, \ldots, k+1$.
Thus $q, s, t \geq 2^{\ell-1}$ and $|s-t|, |s-q|, |t-q|< 2^{\ell-1}$. If $\min\{q, s, t\} \in [2^{\ell'}, 2^{\ell'+1}-1]$ for some $\ell'\geq \ell-1$, 
then $q, t, s \in [2^{\ell'}, 2^{\ell'+1}-1+2^{\ell-1}]$, i.e, $\Delta_1(q), \Delta_1(t), \Delta_1(s) \in \{\ell', \ell'+1\}$. Thus two of $q,s,t$ have the same color and therefore $X_{alt}^{**}$ is not rainbow.   This in particular implies that pattern~(ii) does not hold.
\end{itemize}
Therefore, pattern~(iii) may not be omitted for $k \geq 3$ without violating the theorem.
\end{proof}

%%%%%%%%%%%%%%%%%%%%%%%%%%%%%%%%%%%%%%%%%%%%%%%%%%%
%%%%%%%%%%%%%%%%%%%%%%%%%%%%%%%%%%%%%%%%%%%%%%%%%%%

\vskip 1cm

\section{Concluding remarks} \label{conclusions}

In this paper, we derived canonical Ramsey-type theorems for linear combinations of integers corresponding to sums or alternating sums and respective linear systems. Our results present sets of three or five  unavoidable patterns. In addition we show that each of these patterns is necessary.  It remains a wide question whether analogous statements could be made for general linear combinations. \\

There is a canonical version of the Milliken-Taylor Theorem  for colorings of $k$-element sets of sets, see ~ Lefmann ~\cite{lef8}. For the case of $k=2$ there are $26$ canonical patterns. Using this, we can show for example, that for sequences $(a,a, \ldots , a, 1,1, \ldots 1)$ containing $g$ many $a$'s and $k-g$ many ones, where $g(a-1) + k=0$ or $g(a-1) + k = 1$ and an integer $a < 0$, the set  $X \cup \{ ax_{i_1} +ax_{i_2} + \cdots +ax_{i_g} + x_{i_{g+1}} + x_{i_{g+2}} +\cdots + x_{i_{k}}: ~ 1 \leq i_1 < i_2 < \cdots < i_k  \}$, where $X=\{ x_1 < x_2 < \cdots \}$ is infinite, 
 is colored according to one of at most $26$ canonical patterns. Indeed, for the case $g(a-1) + k=1$ the number of patterns reduces to at most $15$.\\

However, we can only show necessity of at least three patterns for the  statements above.
This brings us to the following:\\

{\bf Open question:}~ For which sequences  $(a_1, a_2,\ldots, a_k)$ of  non-zero integers  is there a constant number, independent of $k$,  of unavoidable patterns so that in any coloring of  positive integers, there is an infinite set $X=\{x_1<x_2<\cdots\}$ such that 
$X \cup \{ a_1x_{i_1} +a_2x_{i_2} + \cdots + a_k x_{i_{k}}: ~ 1 \leq i_1 < i_2 < \cdots < i_k \}$ satisfies one of these patterns?
For which such sequences $(a_1, a_2,\ldots, a_k)$ can we guarantee five unavoidable patterns?\\

 We see that if $a_k<0$ then by looking at $x_{i_k}$ large enough compared to $x_{i_{k-1}}$, the sum  $a_1x_{i_1} +a_2x_{i_2} + \cdots + a_k x_{i_{k}}$ is negative. However since we consider only positive integers this case is ill-defined. From the remarks after 
   Theorem~\ref{ghl} we see that for situations with $a_k>1$,  we cannot guarantee the monochromatic pattern 
   for colorings with only finitely many colors.   These and some other indications make us believe that the problem is reducible to $a_k=1$ and that in case of  
 $\sum_{i=1}^k a_i = 0$ or  $\sum_{i=1}^k a_i = 1$  there are five unavoidable patterns.


\begin{thebibliography} {999}

%\bibitem{B} B.~Barber,  N.~Hindman,  I.~Leader, and   D.~Strauss,
%Partition regularity without the columns property. (English summary) 
%Proc. Amer. Math. Soc. 143 (2015), no. 8, 3387--3399. 
\bibitem{Deuber73} W.~Deuber, Partitionen und lineare Gleichungssysteme,
 Mathematische Zeitschrift 133, 1973, 109--123.
%\bibitem{DGHS} W.~Deuber, D.~S.~Gunderson, N.~Hindman, and D.~Strauss,
% Independent finite sums for $K_m$-free graphs, Journal of
% Combinatorial Theory Series A 78, 1997,
% 171--198.
\bibitem{ER50} P.~Erd\H{o}s and R.~Rado, A combinatorial theorem, Journal
of the London Mathematical Society 25(4), 1950, 249--255.
%\bibitem{GLPR1} D.~S.~Gunderson, I.~Leader, H.~J.~Pr�mel, and V.~R�dl,
% Independent arithmetic progressions in clique-free graphs
% on the natural numbers, Journal of Combinatorial Theory Series A 93, 2001,
% 1--17. 
 \bibitem{ghl} D.~S.~Gunderson, N.~Hindman, and H.~Lefmann, Some partition theorems for infinite and finite matrices, Integers 14, 2014, A 12.
%\bibitem{GLPR2} D.~S.~Gunderson, I.~Leader, H.~J.~Pr�mel, and V.~R�dl,
% Independent Deuber sets in graphs on the natural numbers, Journal of
% Combinatorial Theory Series A 103, 2003, 305--322. 
\bibitem{hindman} N.~Hindman, Finite sums from sequences within  cells of a partition of $N$, Journal of Combinatorial Theory Series A 17, 1974, 1--11.
\bibitem{hindman2} N.~Hindman,  Ultrafilters and Combinatorial Number Theory,
Lecture Notes in Mathematics 751, Springer, 1979, 119--184.
%\bibitem{hindman-leader-straus2003}
%I.~Leader, N.~Hindman, and D.~Strauss, Separating Milliken-Taylor systems
%with negative entries, Proc.\ Edinburgh Math.\ Soc.\ 46, 2003,
%45--61.
\bibitem{hindman_leader} N~Hindman and I.~Leader, Image partition regular matrices, Combinatorics, Probability \& Computing  2, 1993, 437--463.
\bibitem{lef8}  H.~Lefmann, A canonical version for partition regular systems of linear equations, Journal of Combinatorial Theory  Series A 41, 1986, 95--104.
\bibitem{asc} H.~Lefmann, Canonical partition relations for ascending families of finite sets, Studia Scientiarum Mathematicarum Hungarica 31, 1996,  361--374.
\bibitem{milliken} K.~Milliken, Ramsey's theorem with sums or unions,
Journal of Combinatorial Theory Series A 18, 1975, 276--290.
\bibitem{Rado33} R.~Rado, Studien zur Kombinatorik, Mathematische
Zeitschrift 36, 1933, 424--480.
\bibitem{Rado69} R.~Rado, Some Partition Theorems, in: Colloquia Mathematica
 Societatis Janos Bolyai 4. Combinatorial Theory and its Applications,
Balatonfured, Hungary, North Holland, 1969, 929--935.
\bibitem{ramsey} F.~P.~Ramsey, On a problem of formal logic,  
Proceedings of the London Mathematical Society 2, 30, 1930, 264--286. 
\bibitem{taylor} A.~Taylor, A canonical partition relation for finite
subsets of $\omega$, Journal of Combinatorial Theory Series A 21, 1976, 137--146.
\end{thebibliography}
\end{document}